\documentclass[10pt]{amsart}
\usepackage{geometry}
\usepackage[foot]{amsaddr}
\usepackage{amssymb,amsopn,bm}
\usepackage{graphicx,float}
\usepackage{enumerate,array,multirow,bbm,xcolor}
\usepackage{aliascnt,stmaryrd}

\newcommand{\R}{\mathbb{R}}
\newcommand{\C}{\mathbb{C}}
\newcommand{\Z}{\mathbb{Z}}
\newcommand{\Stau}{S_\tau}
\newcommand{\vxi}{{\bm{\xi}}}
\newcommand{\vx}{{\bm{x}}}

\newcommand{\rmd}{\mathrm{d}}
\newcommand{\vep}{\varepsilon}
\newcommand{\vphi}{\varphi}

\DeclareMathOperator*{\esssup}{ess\,sup}

\usepackage[hidelinks]{hyperref}
\usepackage[capitalize,nameinlink,noabbrev]{cleveref}
\crefname{equation}{}{}

\theoremstyle{plain}
\newtheorem{theorem}{Theorem}[section]

\newaliascnt{lemma}{theorem}
\newtheorem{lemma}[lemma]{Lemma}
\aliascntresetthe{lemma}

\newaliascnt{proposition}{theorem}
\newtheorem{proposition}[proposition]{Proposition}
\aliascntresetthe{proposition}

\newaliascnt{corollary}{theorem}

\aliascntresetthe{corollary}

\theoremstyle{remark}
\newaliascnt{remark}{theorem}
\newtheorem{remark}[remark]{Remark}
\aliascntresetthe{remark}

\numberwithin{equation}{section}
\numberwithin{figure}{section}
\numberwithin{table}{section}

\begin{document}

	
	\title[EWI for NLSE with highly singular potential]{Optimal error bounds on the exponential wave integrator for nonlinear Schr\"odinger equations with highly singular potential}
	
	\author[W. Bao]{Weizhu Bao}
	\address{Department of Mathematics, National University of Singapore, Singapore 119076}
	\email{matbaowz@nus.edu.sg}
	
	\author[C. Wang]{Chushan Wang}
	\address{Department of Statistics and Committee on Computational and Applied Mathematics, University of Chicago, Chicago, IL 60637}
	\email{chushanwang@uchicago.edu}
	
	\author[Y. Wu]{Yifei Wu}
	\address{School of Mathematical Sciences, Nanjing Normal University, Nanjing 210046, China}
	\email{yerfmath@gmail.com}
	
	\begin{abstract}
		We establish error estimates of the first-order exponential wave integrator (EWI) for the nonlinear Schr\"odinger equation (NLSE) with a highly singular potential in $\R^d$ with $1 \leq d \leq 3$. Our results deal with singular potentials in $L^p_\text{\rm loc}(\R^d)$ with $p > \frac{d}{2}$ and $p \geq 1$, which is (almost) the weakest regularity of the potential required by the well-posedness of the NLSE. First, for $L^p_\text{loc}$-potentials with $p>2$, we establish an optimal first-order $L^2$-norm convergence for the EWI, with the convergence order slightly reduced to $1^-$ when $p=2$. To the best of our knowledge, the optimal first-order convergence for the three-dimensional $L^2$-potential is for the first time in the literature. The optimality of such an error bound is two-fold: (i) the first-order $L^2$-norm convergence is optimal for the EWI (and its higher-order versions) under the given $L^2$-regularity assumption on the potential, and (ii) to achieve the first-order $L^2$-norm convergence for the EWI, such an assumption is optimally weak. For more singular potentials in $L^p_\text{\rm loc}(\R^d)$ with $\frac{d}{2} < p < 2$ and $p \geq 1$, we prove that the $L^2$-norm convergence is (almost) of $(1 - \alpha)$-order when $d=1, 2$, and of $(1 - \frac{3}{2}\alpha)$-order when $d=3$, where $\alpha := d(1/p - 1/2)$ when $d =1,2,3$, $p>1$ and $\alpha := \frac{1}{2}^+$ when $d=1$, $p=1$. Notably, this result pushes the error estimate to the threshold regularity of the potential that matches the threshold regularity for the well-posedness of the NLSE, which is also for the first time. Two main ingredients are adopted in the proof: (i) the use of discrete space-time Lebesgue spaces together with discrete Strichartz estimates to establish the stability of the numerical scheme, and (ii) the use of normal form transformation and frequency decompositions to obtain optimal error bounds. 
	\end{abstract}
	
	\maketitle
	
	
	\vspace{0.5em}
	
	
	\section{Introduction}
	We consider the Cauchy problem of the following (nonlinear) Schr\"odinger equation (NLSE) with a singular potential in $\R^d$ with $1 \leq d \leq 3$: 
	\begin{equation}\label{NLSE}
		\left\{
		\begin{aligned}
			&i \partial_t \psi(\vx, t) = -\Delta \psi(\vx, t) + V(\vx) \psi(\vx, t) + \beta |\psi(\vx, t)|^2\psi(\vx, t), \quad t>0, \quad \vx \in \R^d, \\
			&\psi(\vx, 0) = \psi_0(\vx), \quad \vx \in \R^d, 
		\end{aligned}
		\right.
	\end{equation}
	where $V = V(\vx)$ is a real-valued potential that may be singular and unbounded and $\beta \in \R$ is a given constant. We make the following assumptions on the potential $V$:  
	\begin{equation}\label{eq:V_assumption}
		V \in L^p(\R^d) + L^\infty(\R^d) := \{ aV_1 + bV_2 \,|\, V_1 \in L^p(\R^d), \ V_2 \in L^\infty(\R^d), \ a, b \in \R \}, \quad  \frac{d}{2} < p < \infty, \quad p \geq 1. 
	\end{equation}
	Here, $p = \frac{d}{2}$ is a critical regularity of the potential $V$ in the sense of scaling: under the natural scaling of the Schr\"odinger equation $V(\vx) \rightarrow \lambda^2 V(\lambda \vx)$, the $L^\frac{d}{2}$-norm is preserved. It is also a critical regularity of potential such that the potential $V$ as a multiplication operator is relatively (form) bounded with respect to the Laplacian \cite{book_Teschl}. This fact also plays an essential role in our error estimates for the three-dimensional (3D) case. Notably, \cref{eq:V_assumption} is (up to the end-points) the weakest regularity that can ensure the well-posedness of the NLSE \cref{NLSE} \cite{wu2025_singular1d,wu2025_singularhd}. We also remark here that we focus on the singular potential with local singularity, which is locally unbounded. This is different from the works in the literature concerning smooth confining potentials that are unbounded at infinity \cite{remi2025,splitting_singular_focm2025}. 
	
	The (nonlinear) Schr\"odinger equation arises from various physical applications, such as quantum physics and chemistry, laser beam propagation, deep water waves, plasma physics, and Bose-Einstein condensates \cite{BEC_book,book_NLSE_Sulem,book_NLSE_Gadi,review_2013,ESY}. In many of these applications, the potential $V$ may be singular or very rough. The most fundamental and important example is the Coulomb potential in 3D with $V(\vx) = -\frac{Z}{|\vx|}$ for some $Z > 0$, which models the Coulomb force at the quantum mechanical level \cite{SE}. Other types of singular or rough potential also arise frequently in quantum mechanics \cite{singular_poten_1964,singular_poten_PR,singular_poten_RevModPhys,defect}, and in the model of wave propagation through disordered media \cite{Anderson,disorder_science,anderson_transverse}. 
	
	In addition to its strong physical relevance, the NLSE \cref{NLSE} with singular potential itself is of significant mathematical interest, presenting substantial challenges in both the analytical study and the numerical approximation. From the analytical point of view, a fundamental question is, under sufficient smooth initial data $\psi_0$, how does the regularity of the potential influences the regularity of the solution. In \cite{cazenave2003}, it is proved that when $V \in L^2(\R^d) + L^\infty(\R^d)$, the NLSE \cref{NLSE} is locally well-posed in $H^2$. Very recently, sharp well-posedness results are obtained, which give a complete characterization of the regularity of the solution under general singular potentials on the whole space $\R^d $ \cite{wu2025_singular1d,wu2025_singularhd} and on the one-dimensional torus \cite{zhao2024}. To be precise, it is proved that for singular potential $V$ satisfying \cref{eq:V_assumption}, the NLSE \cref{NLSE} is sharply well-posed in $H^{2-\alpha}$, where $\alpha = \alpha(p)$ is defined as
	\begin{equation}\label{eq:alpha_def}
		\alpha := \left\{
		\begin{aligned}
			&d\left( \frac{1}{p} - \frac{1}{2} \right), && \frac{d}{2} < p \leq 2, \  d=2, 3, \\
			&\frac{1}{p} - \frac{1}{2}, && 1 < p \leq 2, \  d=1, \\
			&\frac{1}{2}^+, && p = 1, \  d = 1,  
		\end{aligned}
		\right.
		\quad
		\text{ and we note that }   
		\alpha \in \left\{
		\begin{aligned}
			&[0, \frac{1}{2}^+], &&d=1, \\
			&[0, 1), &&d=2, \\
			&[0, \frac{1}{2}), && d=3. 
		\end{aligned}
		\right.
	\end{equation}
	The sharpness lies in the fact that $H^{2-\alpha}$ is the maximal regularity that can be attained by the solution under an $L^p_\text{\rm loc}$-potential. Moreover, it is proved that the equation is ill-posed in any $H^s$-space $(s \in \R)$ for $L^p_{\rm loc}$-potential with $p<\frac{d}{2}$ or $p<1$. We remark here that as $2-\alpha > d/2$, we always have $\psi(t) \in L^\infty(\R^d)$. 
	
	Although much progress has been made in the analytical study of the NLSE with singular potential, most of the numerical theories of the NLSE are still restricted to bounded and sufficiently smooth potentials (usually assumed to be $H^2$ or $H^1$) \cite{jahnke2000,bao2023_semi_smooth,tree2,ExpInt,splitting_acta}. Some recent works have successfully weakened the regularity requirement on the potential to $L^\infty$, allowing for rough but bounded potentials \cite{henning2017,bao2023_EWI,bao2023_improved,bao2023_sEWI,bao2024,lin2024}. In particular, under the assumption of $L^\infty$-potential, the optimal error bounds have been established at the full-discretization level, for some standard time integrators, such as the time-splitting methods \cite{bao2023_improved} and the exponential wave integrators (EWI) \cite{bao2023_EWI,bao2023_sEWI}, coupled with Fourier spectral based spatial discretizations \cite{bao2024,lin2024}. However, the extension of these results to the singular case for $L^p_{\rm loc}$-potential with $p < \infty$ is not straightforward, and the existing results on singular potentials are quite limited. 
	For the Schr\"odinger equation (i.e. \cref{NLSE} with $\beta = 0$) with Coulomb potential in 3D (which is in $L^{3^-}_{\rm loc}$), the time-splitting method is analyzed, for the eigenstates initial data in \cite{PRR} and for general initial data in \cite{splitting_singular_focm2025,fang2025}, based on commutator estimates \cite{jahnke2000,lubich2008}. Their results show that (roughly) $\frac{1}{4}$-order convergence in $L^2$-norm can be achieved by time-splitting methods under 3D Coulomb potential, and the convergence order deteriorates to $0$ as the potential becomes more singular, e.g., $ V \in L_{\rm loc}^2$. Compared to the known $H^2$-regularity of the exact solution, such convergence order, though optimal (at the time semidiscretization level) as confirmed by the numerical results, is rather low. In \cite{zhao2024}, a novel exponential-type integrator is proposed and analyzed for the one-dimensional (1D) NLSE on the torus with (highly) singular potential that may not even lie in $L^2_{\rm loc}$. Owing to the special construction of the method, they are able to prove in the $L^2$-norm $\frac{3}{4}$-order convergence under $L_{\rm loc}^2$-potential and $\frac{1}{4}$-order convergence under $L_{\rm loc}^1$-potential. Very recently, with the use of discrete Strichartz estimates, a filtered version of the standard EWI (the filtered exponential Euler method) is analyzed in \cite{bao2025_singular}, and (nearly) first-order $L^2$-norm convergence is obtained for singular potentials that lie in $L^2_{\rm loc}$ in 1D and 2D, and $L_{\rm loc}^{\frac{7}{4}}$ in 3D (hence covers the 3D Coulomb case), and the convergence order reduced to $\frac{3}{4}$-order for $L^2_{\rm loc}$ potentials in 3D. Such convergence estimates have significantly improved the existing results, indicating that the standard exponential integrators would have significant advantages in simulating the NLSE with singular potentials. Interestingly, however, the convergence order reduction in 3D for $L^2_{\rm loc}$-potentials is not observed in the numerical experiments \cite{bao2025_singular}. 
	
	Considering the significantly better performance of the EWI in both the theoretical analysis and the practical computations, in this work, we continue to focus on the (filtered) exponential Euler method: 
	Choosing $\tau > 0$ to be the time step size, and denoting the time steps as $t_n = n \tau$ for $n = 0, 1, \cdots$,  with $\psi_\tau^n$ being the approximation to $\psi(\cdot, t_n)$, the filtered EWI reads \cite{bao2025_singular}
	\begin{equation}\label{eq:EWI_scheme}
		\begin{aligned}
			&\psi_\tau^{n+1} = e^{i \tau \Delta} \psi^n_\tau - i \tau \vphi_1(i \tau \Delta) \Pi_\tau (V \psi_\tau^n + \beta |\psi_\tau^n|^2 \psi_\tau^n), \quad n \geq 0, \\
			&\psi^0_\tau = \Pi_\tau \psi_0, 
		\end{aligned}
	\end{equation}
	where $\vphi_1(z) = \frac{e^z - 1}{z}$ for $z \in \C$, and $\vphi_1(i \tau \Delta)$ is the Fourier multiplier with symbol $\vphi_1(- i \tau |\vxi|^2)$, and $\Pi_\tau$ is a smooth filter function in the Fourier space defined in \cref{eq:Pitau_def}. 
	
	Our first contribution is an improvement of the error analysis on \cref{eq:EWI_scheme} in \cite{bao2025_singular} for $L^p_\text{\rm loc}$-potentials with $p>2$, extending the optimal first-order $L^2$-norm convergence to 3D. Choose a final time $0<T<T_\text{max}$ with $T_\text{max}$ being the maximal existing time of the exact solution to \cref{NLSE}. 
	
	\begin{theorem}\label{thm:main1}
		Assume that $V \in L^{p}(\R^d) + L^\infty(\R^d)$ with $p > 2 $ and $\psi_0 \in H^2(\R^d)$ for $d = 1, 2, 3$. There exists $\tau_0>0$ sufficiently small such that when $0<\tau\leq\tau_0$, we have
		\begin{equation*}
			\| \psi(t_n) - \psi^n_\tau \|_{L^2} \lesssim \tau, \qquad 0 \leq n \leq T/\tau. 
		\end{equation*}
	\end{theorem}
	When $d=3$, the above result improves the previous $\frac{3}{4}$-order convergence to optimal first order. The optimality of the first-order $L^2$-norm convergence is two-fold: (i) the convergence order is optimal for the EWI under the assumption of $L_{\rm loc}^2$-potentials and (ii) the assumption of $L_{\rm loc}^2$-regularity of the potential is the weakest regularity for the EWI to achieve first-order convergence due to the aforementioned sharp $H^2$ well-posedness result. 
	
	We further extend the error analysis to highly singular potentials in $L^p_{\rm loc}$ with $d/2 < p \leq 2$ and $ p \geq 1$. 
	\begin{theorem}\label{thm:main2}
		Assume that $V \in L^{p}(\R^d) + L^\infty(\R^d)$ with $ d/2 < p \leq 2, \ p \geq 1$ and $\psi_0 \in H^{2 - \alpha}(\R^d)$ with $\alpha$ given in \cref{eq:alpha_def} for $d=1, 2, 3$. There exists $\tau_0>0$ sufficiently small such that when $0<\tau\leq\tau_0$, we have
		\begin{equation*}
			\| \psi(t_n) - \psi^n_\tau \|_{L^2} \lesssim
			\left\{
			\begin{aligned}
				&\tau^{1 - \alpha}, && d=1, \  1 < p \leq 2,  \\
				&\tau^{\frac{1}{2}^-}, && d=1, \  p = 1, \\
				&\tau^{1^- - \alpha}, &&d=2, \ 1 < p \leq 2, \\
				&\tau^{1^- - \frac{3}{2}\alpha}, &&d=3, \ \frac{3}{2} < p \leq 2,  
			\end{aligned}
			\right.
			\qquad 0 \leq n \leq T/\tau.  
		\end{equation*} 
	\end{theorem}
	
	Recalling the sharp well-posedness theory of the NLSE \cite{wu2025_singular1d,wu2025_singularhd}, we have established the error estimate at the same threshold regularity of the potential as is needed by the well-posedness of the equation, which appears to be very rare in the numerical analysis literature. The optimality of the convergence order in 1D is confirmed by the numerical results, while it remains unclear and seems difficult to check numerically whether the convergence order in 2D and 3D is optimal. 
	
	\begin{remark}[$H^1$-norm error estimates]
		Similar to \cite{bao2025_singular}, due to the filteration of the numerical solution, the error bound in the $H^1$-norm can be obtained directly from the corresponding one in the $L^2$-norm with the order reduced by $\frac{1}{2}$ if the $L^2$-norm error bound has an order greater than $\frac{1}{2}$. If the $L^2$-norm error bound is first-order, then the half-order $H^1$-norm error bound is optimal (for the EWI) with respect to the $H^2$-regularity of the exact solution. 
	\end{remark}   

	We explain the main idea of the proof in the following. The first difficulty in establishing error estimates under singular potentials comes from the loss of the usual $L^2$-stability of the numerical flow. Consider a general exponential-type integrator with the numerical flow for one time step given by $\Phi^\tau(\phi) := e^{i \tau \Delta}\phi + \tau \widetilde{\Phi}^\tau(\phi)$ where $\tau>0$ is the time step size. The standard $L^2$-stability of the numerical flow takes the form: for $\phi \in L^2$, 
	\begin{equation*}
		\| \Phi^\tau(\phi) \|_{L^2} \leq \| \phi \|_{L^2} + \tau \| \widetilde{\Phi}^\tau(\phi) \|_{L^2} \leq (1 +  C\tau) \| \phi \|_{L^2}, 
	\end{equation*}
	where $C:= \| \widetilde{\Phi}^\tau \|_{L^2  \rightarrow L^2}$ is the stability constant that is independent of $\tau$ but may depend on $\phi$ in nonlinear problems. For unbounded singular potential $V \in L^p \setminus L^\infty$, the flow $\widetilde{\Phi}^\tau$ may fail to map $L^2$ into $L^2$ or the operator norm $\| \widetilde{\Phi}^\tau \|_{L^2 \rightarrow L^2}$ may fail to be bounded uniform in $\tau$, simply due to the fact that given $\phi \in L^2$, $V \phi \not \in L^2$ in general if $V \not \in L^\infty$. To overcome the loss of stability, instead of working under $L^2$-spaces, we establish the error estimate in some proper space-time spaces defined in \cref{eq:norm_discrete} with the use of discrete Strichartz estimates (see \cref{lem:dS1,lem:dS2,lem:dSc}) \cite{ignat2011,choi2021,LRI_error_1,su2022}. In fact, the standard $L^2$-stability can also be understood as the stability in the space-time space $l^\infty_\tau L^2_\vx$. In \cite{bao2025_singular}, the discrete Strichartz estimates have been applied to dealing with singular $L^2$-potentials, and the same idea can be generalized to the current case for highly singular $L^p$-potentials with $p<2$. 
	
	Another difficulty comes from the significant order reduction in the estimate of errors from the temporal discretization and the spatial discretization (filteration), which involves the singular potential. In addition to the use of discrete Strichartz estimates, a key observation used in the estimate of temporal convergence is that the temporal error in the EWI depends on the space-time estimates of the time derivative of the exact solution, which is available according to the well-posedness in \cite{wu2025_singular1d,wu2025_singularhd}: for any admissible pair $(q, r)$ (see \cref{eq:admissible}), 
	\begin{equation}\label{eq:psi_spacetime}
		\| \langle \nabla \rangle^{-\alpha} \partial_t \psi \|_{L^q_t L^r_\vx} \lesssim 1, 
	\end{equation}
	where $\langle \nabla \rangle$ is the Fourier multiplier with symbol $(1+|\vxi|^2)^\frac{1}{2}$. With \cref{eq:psi_spacetime}, somewhat surprisingly, we are able to establish optimal temporal convergence in all the cases (see \cref{prop:E2}). 
	
	However, the situation for the spatial filteration error is more involved. The main difference is that for the spatial convergence, the error thus depends on the spatial derivatives of the exact solution, which admits no non-trivial space-time estimates. This can be seen from the equation \cref{NLSE}: applying $|\nabla|^{-\alpha}$ on both sides of \cref{NLSE} and assuming $\beta = 0$ and $V \in L^p$ for simplicity,  
	\begin{equation}
		|\nabla|^{2-\alpha} \psi(t) = |\nabla|^{-\alpha} \partial_t \psi(t) - |\nabla|^{-\alpha} (V \psi(t)). 
	\end{equation}
	Though we have \cref{eq:psi_spacetime} for $\partial_t \psi(t)$, since $V \in L^p$, we can only control the $L^2$-norm of the last term through $ \| |\nabla|^{-\alpha} (V \psi(t)) \|_{L^2} \lesssim \| V \psi(t) \|_{L^p} \leq \| V \|_{L^p} \| \psi(t) \|_{L^\infty} $, where we use the Sobolev embedding. Hence, there is no space-time estimates except $L_t^\infty L_\vx^2$ for $|\nabla|^{2-\alpha} \psi(t)$. As a result, the spatial error established in \cref{prop:E1_pl2} is of lower order compared to the temporal error when $\alpha > 0$ (or, equivalently, $p \leq 2$).
	
	In addition to the loss of spatial convergence due to the lack of space-time estimates of the spatial derivative of $\psi$, there is another source of loss in 3D due to the end-point of the admissible pairs: the end-points in 1D and 2D are $(4, \infty)$ and $(2^+, \infty^-)$, while the end-point in 3D is $(2, 6)$. As a result, when using the discrete Strichartz estimate \cref{lem:dSc}, one only needs to control the error in (roughly) $L_\vx^{1}$-norm in 1D and 2D, then needs to control the error in $L_\vx^{\frac{6}{5}^+}$-norm in 3D. In fact, one of the main contributions of the current paper is to establish an improved estimate of this spatial filteration error in 3D, which is done in \cref{prop:R1,prop:R2} with the use of normal form transformation and high-low frequency decomposition. In particular, the estimates established here improve the existing estimates for 3D $L^{2^+}$-potential established in \cite{bao2025_singular} to optimal first-order, and are further generalized to more singular potentials.

	Next, we briefly explain the idea of the estimate for the filteration error in 3D. The details can be found in \cref{sec:3}. Again, for simplicity, we assume $\beta = 0$ and $V \in L^p(\R^3)$. As we have already mentioned, there are no space-time estimates of $|\nabla|^{2-\alpha} \psi$ in addition to the trivial one in $L^\infty_t L^2_{\vx}$. Hence, to avoid directly estimating $|\nabla|^{2-\alpha} \psi$, we can make use of the equation by noting that
	\begin{equation}
		-\Delta \psi = i \partial_t \psi - V \psi. 
	\end{equation}
	It turns out that the two terms on the RHS are more favorable. The benefit of transferring to $\partial_t \psi$ lies in that it admits space-time estimates as presented in \cref{eq:psi_spacetime} while the advantage of $V \psi$ lies essentially in the fact that $V$ (under our assumption) is a lower order perturbation of $-\Delta$. In fact, since $V \in L^p(\R^3)$ with $p > 3/2$, we have, for any $\phi = \phi(\vx)$ and $1 \leq  q \leq p$, 
	\begin{equation}\label{eq:V_order}
		\| V \phi \|_{L^q} \leq \| V \|_{L^p} \| \phi \|_{L^\frac{pq}{p-q}} \lesssim \| V \|_{L^p} \| \, |\nabla|^{\frac{3}{p}^+} \phi \|_{L^q}.  
	\end{equation}
	Then the multiplication operator $V$ is relatively bounded by $| \nabla |^{\frac{3}{p}^+}$ (in the sense of \cref{eq:V_order}), which is of lower order than $-\Delta = |\nabla|^2$ when $p > 3/2$. Hence, by replacing $-\Delta \psi$ with $V \psi$, intuitively, one can gain $2 - \frac{3}{p} = \frac{1}{2} - \alpha$-order derivative. For $L^2$-potentials, we have $\alpha=0$, and thus we gain $\frac{1}{2}$-order derivative by this substitution, which fully compensates the upgrade from $L^1$-norm to $L^\frac{6}{5}$-norm in the 3D estimate. To indeed make use of such gain of regularity of $\frac{1}{2}-\alpha$-order from replacing $-\Delta \psi$ by $V \psi$, we (ideally) hope the following
	\begin{equation}
		|\nabla|^{\frac{1}{2} - \alpha} (V \psi) \quad \text{be controlled by} \quad V  |\nabla|^{\frac{1}{2} - \alpha} \psi.  
	\end{equation}
	However, it is impossible as $|\nabla|^{\frac{1}{2}-\alpha}$ and $V$ do not commute (in fact, the former is not even well-defined for $V$ merely in $L^p_\text{loc}$ with $p > 3/2$). Instead, to use such gain of the regularity, we employ the techniques of (partial) normal form transformation and high-low frequency decomposition for $V \psi$ such that: (i) when $\psi$ is of higher-frequency in $V \psi$, the spatial derivatives can effectively pass through $V$ and fall on $\psi$, and (ii) when $\psi$ is of lower-frequency in $V \psi$, by using a partial normal form transformation, the spatial derivative is transferred to a composition of a time derivative and a negative-order spatial derivative such that both of them can be applied directly to $\psi$ across $V$. 
	
	The rest of the paper is organized as follows. In \cref{sec:2}, we present some preliminary results. \cref{sec:3} is devoted to the estimate of the spatial filteration error. The error estimates of the EWI is then established in \cref{sec:4}. The extension to the nonlinear case is discussed in \cref{sec:nonlinear}, and the numerical results are demonstrated in \cref{sec:numer}. Finally, some conclusions are drawn in \cref{sec:conclusion}. 
	
	\subsection*{Notation.} We use $| \cdot |$ for the Euclidean norm in $\R^d$. For $s \in \R$, 
	we denote by $|\nabla|^s$ and $\langle \nabla \rangle^s$ the Fourier multipliers with symbols $|\vxi|^s$ and $ \langle \vxi \rangle^s:=(1+|\vxi|^2)^\frac{s}{2} $, respectively. 
	For $a \in \R$, we write $a^-$ to denote $a-\vep$ for an arbitrarily small $\vep>0$. For any $p \in [1, \infty]$, $p' \in [1, \infty]$ is the H\"older conjugate of $p$ such that $\frac{1}{p}+\frac{1}{p'} =1$. We denote by $C$ a generic positive constant independent of the time step size $\tau$, and by $C(\alpha)$ a generic positive constant depending on the parameter $\alpha$. The notation $A \lesssim B$ is used to represent that there exists a generic constant $C>0$, such that $|A| \leq CB$. 
	
	
	\section{Preliminary results}\label{sec:2}
	In this section, we recall some preliminary results to be used in the proof. 

	\subsection{Discrete Strichartz estimates}	
	In this subsection, we recall the discrete-in-time Strichartz estimates. Let $\phi:\R^d \times \R \rightarrow \C$. For any interval $I \subset \R$, we consider the space-time Lebesgue spaces $L^q(I; L^r(\R^d))$ with $1 \leq q, r \leq \infty$ equipped with the following norms
	\begin{equation}\label{eq:norm_continuous}
		\| \phi \|_{L^q(I; L^r)}:=
		\left\{
		\begin{aligned}
			&\left(\int_I \| \phi(\cdot, t) \|_{L^r}^q \rmd t \right)^\frac{1}{q}, &&1 \leq q < \infty, \\
			&\esssup_{t \in I} \| \phi(\cdot, t) \|_{L^r}, && q = \infty. 
		\end{aligned}
		\right. 
	\end{equation}
	At the time discrete level, for any $I \subset \R$ and $\tau > 0$, we define, for a sequence $(\phi^k(\vx))_{k \in \mathbb{Z}}$, 
	\begin{equation}\label{eq:norm_discrete}
		\| \phi \|_{l^q_\tau(I; L^r)} := 
		\left\{
		\begin{aligned}
			&\left(\tau \sum_{\tau k \in I} \| \phi^k \|_{L^r}^q\right)^\frac{1}{q}, &&1 \leq q < \infty, \\
			&\sup_{\tau k \in I} \| \phi^k \|_{L^r}, && q = \infty. 
		\end{aligned}
		\right.
	\end{equation}
	Introduce the filter function $\Pi_\tau$ as 
	\begin{equation}\label{eq:Pitau_def}
		(\Pi_\tau \phi)(\vx) = \frac{1}{(2 \pi)^d} \int_{\R^d} \chi(\tau^{\frac{1}{2}} {\vxi}) \widehat{\phi}(\vxi) e^{i \vx \cdot \vxi} \rmd \vxi, \quad \vx \in \R^d, 
	\end{equation}
	where $\chi \in C^\infty(\R^d)$ satisfies $\chi(\vx) \equiv 1$ for $\vx \in B(0, 1)$ and is supported on $B(0, 2) $, and $\widehat \phi$ is the Fourier transform of $\phi$, which, when $\phi \in L^1(\R^d)$, is given by, 
	\begin{equation}
		\widehat{\phi}(\vxi) = \int_{\R^d} \phi(\vx) e^{- i \vx \cdot \vxi} \rmd \vx, \qquad \vxi \in \R^d. 
	\end{equation} 
	Note that 
	\begin{equation}\label{eq:Pitau_identity}
		\Pi_\tau \Pi_{\tau/4} = \Pi_{\tau/4} \Pi_\tau = \Pi_\tau. 
	\end{equation}
	Define the filtered Schr\"odinger group
	\begin{equation}\label{eq:S_tau}
		\Stau(t) = e^{i t \Delta}\Pi_{\tau/4}, \quad t \in \R. 
	\end{equation}
	A pair $(q, r) \in [2, \infty] \times [2, \infty]$ is admissible if
	\begin{equation}\label{eq:admissible}
		\frac{2}{q} = d\left(\frac{1}{2} - \frac{1}{r} \right), \qquad (q, r, d) \neq (2, \infty, 2). 
	\end{equation}
	We have the following discrete Strichartz estimates \cite{LRI_error_1}: 
	\begin{enumerate}[(i)]
		\item Let $(q, r)$ be admissible, for all $\phi \in L^2(\R^d)$,  we have
		\begin{equation}\label{lem:dS1}
			\| S_\tau(\cdot) \phi \|_{l_\tau^q(\tau \Z; L^r)} \leq C(d, q) \| \phi \|_{L^2}. 
		\end{equation}
		
		\item Let $(q, r)$, $(\tilde q, \tilde r)$ be admissible and $(q, \tilde q) \neq (2, 2)$, for all $f \in l_\tau^{\tilde q'}(\tau \Z; L^{\tilde r'}(\R^d))$ and $\theta \in [-1, 1]$, we have
		\begin{equation}\label{lem:dS2}
			\left\| \tau \sum_{k = -\infty}^{n-1} S_\tau((n-k+\theta)\tau) f^k \right\|_{l_\tau^q(\tau \Z; L^r)} \leq C(d, q, \tilde q) \| f \|_{l_\tau^{\tilde q'}(\tau \Z; L^{\tilde r'})}. 
		\end{equation}
		
		\item Let $(q, r)$, $(\tilde q, \tilde r)$ be admissible and $(q, \tilde q) \neq (2, 2)$, for all $f \in L^{\tilde q'}(\R; L^{\tilde r'}(\R^d))$, we have
		\begin{equation}\label{lem:dSc}
			\left\| \int_{-\infty}^{n\tau} S_\tau(n\tau - s) f(s) \rmd s \right \|_{l_\tau^q(\tau \Z; L^r)} \leq C(d, q, \tilde{q}) \| f \|_{L^{\tilde q'}(\R; L^{\tilde r'})}. 
		\end{equation}
	\end{enumerate}
	As discussed in \cite[Remark 3.4]{bao2025_singular}, all these estimates can be easily extended to finite time intervals. 
	
	\subsection{Frequency decomposition operators}	
	We introduce some Fourier multipliers related to the frequency decomposition that will be used frequently in the proof. From $\chi$ introduced in \cref{eq:Pitau_def}, define a function $\vphi$ as
	\begin{equation}
		\vphi(\vxi):= \chi(\vxi/2)-\chi(\vxi), \qquad \vxi \in \R^d. 
	\end{equation} 
	Then $\vphi \in C_\text{c}^\infty(\R^d)$ is supported in an annular region satisfying $\vphi(\vxi) = 0$ if $ |\vxi| \leq 1 $ or $|\vxi| \geq 4$. We define, for any real number $N \geq 1$, $P_N$ and $P_{<N}$ as the Fourier multipliers with symbols $\vphi(\vxi/N)$ and $\chi(\vxi/N)$, respectively. Moreover, we set 
	\begin{equation}\label{eq:PN}
		P_{>N} := I -  P_{< N}, \quad P_{\gg N} := P_{> 10N}, 
		\quad P_{\lesssim N} := I - P_{\gg N}. 
	\end{equation}
	In particular, we have the following dyadic decomposition of the identity: 
	\begin{equation}\label{eq:dyadic_decomposition}
		I = \mathrm{Id} = P_{<1} + \sum_{M \geq 1} P_M, 
	\end{equation}
	where $M$ is taken over all dyadic integers $M = 2^j$. We shall adopt this convention whenever $M$ (or $M'$) appears.  
	In addition, we have 
	\begin{equation}
		\Pi_\tau = P_{<N_\tau}, \quad N_\tau := \tau^{-\frac{1}{2}}. 
	\end{equation}
	We recall in the following some standard estimates of these operators \cite{choi2021,wu2025_singularhd}. 
	\begin{lemma}\label{lem:Pi}
		For any $1 \leq p \leq \infty$ and $\phi:\R^d \rightarrow \C$ measurable, we have, for $N \geq 1$, 
		\begin{align}
			&\| P_N \phi \|_{L^p} + \| P_{<N} \phi \|_{L^p} + \| P_{>N} \phi \|_{L^p} \lesssim \| \phi \|_{L^p}, \label{eq:p1} \\
			&\| P_{N} \phi \|_{L^p} + \| P_{<N} \phi \|_{L^p} \lesssim N^{d(\frac{1}{q} - \frac{1}{p})} \| \phi \|_{L^q},  \qquad 1 \leq q \leq p, \label{eq:bernstein}\\
			&\| P_{>N} \phi \|_{L^p} \lesssim N^{-\gamma} \| \, |\nabla|^{\gamma} \phi \|_{L^p}, \qquad \gamma \geq 0, \label{eq:p2} \\
			&\| \langle \nabla \rangle^\gamma P_{<N} \phi \|_{L^p} + \| \, |\nabla|^\gamma P_{<N} \phi \|_{L^p} \lesssim N^\gamma \| \phi \|_{L^p}, \qquad \gamma \geq 0, \label{eq:p3} 
		\end{align}
		where all the constants are independent of $N$ and $\phi$. 
	\end{lemma}
	
	\subsection{Error equation}
	We briefly recall the derivation of the error equation for the EWI for the convenience of the reader. As the main difficulty of the error estimates lies in the handling of singular potentials, we shall omit the nonlinearity (i.e. taking $\beta = 0$ in \cref{NLSE}) for a better illustration. The extension to the NLSE is relatively more standard and is discussed in \cref{sec:nonlinear}. 
	
	Define the error function $e^n: = \Pi_\tau \psi(t_n) - \psi^n_\tau$ for $0 \leq n \leq T/\tau$. Iterating \cref{eq:EWI_scheme}, recalling \cref{eq:S_tau,eq:Pitau_identity}, we obtain 
	\begin{align}\label{eq:duhamel_discrete}
		\psi^n_\tau 
		&= e^{i n \tau \Delta} \psi^0_\tau -i \tau \sum_{k=0}^{n-1} e^{i(n-1-k)\tau\Delta} \vphi_1(i \tau \Delta) \Pi_\tau (V\psi^k_\tau) \notag \\
		&= \Stau(n\tau) \psi^0_\tau -i \tau \sum_{k=0}^{n-1} \Stau((n-1)\tau - k \tau) \vphi_1(i \tau \Delta) \Pi_\tau (V \psi^k_\tau), \qquad 0 \leq n \leq T/\tau. 
	\end{align}
	Applying $\Pi_\tau$ to the Duhamel's formula for the exact solution $\psi$, we get
	\begin{equation}\label{eq:duhamel}
		\Pi_\tau \psi(n\tau) = \Stau(n\tau) \Pi_\tau \psi_0 - i \int_0^{n\tau} \Stau(n \tau - s) \Pi_\tau (V \psi(s)) \rmd s. 
	\end{equation}
	Subtracting \cref{eq:duhamel_discrete} from \cref{eq:duhamel}, we have
	\begin{align}\label{eq:error_eq}
		e^n = \mathcal{E}_0^n + \mathcal{E}_1^n + \mathcal{E}_2^n + Z^n, \quad n \geq 0,  
	\end{align}
	where
	\begin{align}
		&\mathcal{E}_0^n := \Stau(n\tau) (\Pi_\tau \psi_0 - \psi^0_\tau), \label{eq:E0}\\
		& \mathcal{E}_1^n := - i \int_0^{n\tau} \Stau(n \tau - s) \Pi_\tau \left(V (I - \Pi_\tau) \psi(s) \right) \rmd s, \label{eq:E1}\\
		&\mathcal{E}_2^n := - i \int_0^{n\tau} \Stau(n \tau - s) \Pi_\tau (V\Pi_\tau \psi(s)) \rmd s + i \tau \sum_{k=0}^{n-1} \Stau((n-1)\tau - k \tau) \vphi_1(i \tau \Delta) \Pi_\tau (V\Pi_\tau\psi(k\tau)), \label{eq:E2}\\
		&Z^n := - i \tau \sum_{k=0}^{n-1} \Stau((n-1)\tau - k \tau) \vphi_1(i \tau \Delta) \Pi_\tau ( V e^k ). \label{eq:Z_def}
	\end{align}
	Here, $\mathcal{E}_0 $ is the error resulting from the discretization of the initial data, and $\mathcal{E}_1$ and $\mathcal{E}_2$ are the errors introduced from the spatial filteration (or spectral truncation) and the temporal discretization in the EWI \cref{eq:EWI_scheme}, respectively. We have used the short notation $\mathcal{E}$ to denote the sequence $(\mathcal{E}^n)_{n \geq 0}$. 
	
	The error estimate of the EWI \cref{eq:EWI_scheme} then reduces to the estimates of the four terms \cref{eq:E0,eq:E1,eq:E2,eq:Z_def}. For $\mathcal{E}_1$, the estimate is presented in \cref{sec:3}, while the rest are presented in \cref{sec:4}. 
	
	\allowdisplaybreaks
	\section{Estimate of the spatial filteration error for \texorpdfstring{\cref{NLSE} with $\beta = 0$}{NLSE with beta=0}}\label{sec:3}
	In this section, we establish the estimates for $\mathcal{E}_1$ \cref{eq:E1}. We first recall the known regularity results on the exact solution $\psi$. Recall that $T>0$ is a fixed final time. By the well-posedness results in \cite{wu2025_singular1d,wu2025_singularhd}, if $\psi_0 \in H^{2-\alpha}(\R^d)$, we have $\psi \in C([0, T]; H^{2-\alpha}(\R^d)) \cap C^1([0, T]; H^{-\alpha}(\R^d))$. Moreover, we have, for any admissible pair $(q, r)$, 
	\begin{equation}\label{eq:regularity}
		\| \psi \|_{L^\infty([0, T]; H^{2-\alpha})} + \| \langle \nabla \rangle^{-\alpha} \partial_t \psi \|_{L^q([0, T]; L^r)} \lesssim 1. 
	\end{equation} 
		
	For the potential $V$, we assume that
	\begin{equation}\label{eq:V_decomp}
		V(\vx) = V_1(\vx) + V_2(\vx), \qquad \vx \in \R^d, \qquad V_1 \in L^p(\R^d), \quad V_2 \in L^\infty(\R^d). 
	\end{equation}
	In fact, we can also assume that $V_1 \in L^1(\R^d) \cap L^p(\R^d)$ in the above decomposition. We then define some admissible pairs to be used in the proof. Define $(q_1, r_1)$ such that
	\begin{equation}\label{eq:r1_def}
		(q_1, r_1) = \left\{
		\begin{aligned}
			&(4, \, \infty), && d=1, \\
			&(\frac{2}{1-\vep}, \, \frac{2}{\vep}), && d=2, \\
			&(\frac{2}{1-\vep}, \, \frac{6}{1+ 2\vep}), && d =3,
		\end{aligned}
		\right.
	\end{equation}
	where $\vep >0$ can be arbitrarily small. Here, $(q_1, r_1)$ are the end-point admissible pairs for the discrete Strichartz estimates. We further define another admissible pair $(q_0, r_0)$ where $q_0$ is determined by $r_0$ through \cref{eq:admissible} with $r_0$ satisfying
	\begin{equation}\label{eq:r0_def}
		\frac{1}{p} + \frac{1}{r_0} + \frac{1}{r_1} = 1, 
	\end{equation}
	where we have $r_0 \geq 2$ and $r_0 \leq \infty$ if $d=1$, $r_0 < \infty$ if $d=2$, and $r_0 < 6$ if $d=3$ when $\vep$ is sufficiently small. 
	It then follows from the H\"older's inequality that 
	\begin{equation}\label{eq:V1}
		\| V_1 \phi \|_{L^{r_1'}} \leq \| V_1 \|_{L^p} \| \phi \|_{L^{r_0}}, \qquad \phi \in L^{r_0}(\R^d). 
	\end{equation}
	It turns out that, although the EWI \cref{eq:EWI_scheme} is not stable in the space $l^\infty_\tau([0, T]; L^2(\R^d))$ as is discussed in the introduction, it is stable in the space $l^\infty_\tau([0, T]; L^2(\R^d)) \cap l^{q_0}_\tau([0, T]; L^{r_0}(\R^d))$ (see \cref{prop:Z}). 
	
	Then we have the following. 
	\begin{proposition}\label{prop:E1_pg2}
		Assume $p > 2$. When $d=1, 2$, for any admissible pairs $(q, r)$ with $q \neq 2$, we have
		\begin{equation}\label{prop:E11_12}
			\| \mathcal{E}_1 \|_{l^{q}_\tau([0, T]; L^{r})} \lesssim \tau. 
		\end{equation}
		When $d=3$, we have
		\begin{equation}\label{prop:E11_3}
			\| \mathcal{E}_1 \|_{l^{q_0}_\tau([0, T]; L^{r_0})} +  \| \mathcal{E}_1 \|_{l^{\infty}_\tau([0, T]; L^{2})} \lesssim \tau. 
		\end{equation}
	\end{proposition}
	
	\begin{proposition}\label{prop:E1_pl2}
		Assume $\frac{d}{2} < p \leq 2$, $p \geq 1$. When $d=1, 2$, for any admissible pairs $(q, r)$ with $q \neq 2$, we have
		\begin{equation}\label{prop:E12_12}
			\| \mathcal{E}_1 \|_{l^{q}_\tau([0, T]; L^{r})} \lesssim
			\left\{
			\begin{aligned}
				&\tau^{1-\alpha}, && d = 1, \  p > 1, \\
				&\tau^{\frac{1}{2}^-}, && d = 1, \  p = 1, \\
				&\tau^{1^--\alpha}, && d = 2, \  1 < p \leq 2.  
			\end{aligned}
			\right. 
		\end{equation}
		When $d=3$, we have
		\begin{equation}\label{prop:E12_3}
			\| \mathcal{E}_1 \|_{l^{q_0}_\tau([0, T]; L^{r_0})} +  \| \mathcal{E}_1 \|_{l^{\infty}_\tau([0, T]; L^{2})} \lesssim \tau^{1^--\frac{3}{2}\alpha}. 
		\end{equation}
	\end{proposition}
	
	\begin{remark}
		In \cref{prop:E1_pg2,prop:E1_pl2}, compared to the estimates for $d=1,2$ (i.e. \cref{prop:E11_12,prop:E12_12}), the estimates for the 3D case (i.e. \cref{prop:E11_3,prop:E12_3}) are stated for two specific admissible pairs of use, where we note that $(q_0, r_0)$ depends on $p$. The reason is that we use the normal transformation to estimate $\mathcal{E}_1$ in 3D, and the estimates of the resulting boundary terms depend on the specific norm we use. Instead, the estimates \cref{prop:E11_12,prop:E12_12} are obtained by direct use of the discrete Strichartz estimates \cref{lem:dSc}, and thus hold uniformly for all admissible pairs.
	\end{remark}
	
	The remainder of this section is devoted to the proof of \cref{prop:E1_pg2,prop:E1_pl2} for the estimate of $\mathcal{E}_1$, which, by recalling \cref{eq:V_decomp,eq:E1} and noting $I - \Pi_\tau = P_{>N_\tau}$, can be decomposed as
	\begin{equation}\label{eq:E1_decomp}
		\mathcal{E}_1 = \mathcal{E}_{1, 1} + \mathcal{E}_{1, 2}, \qquad \mathcal{E}_{1, j}^n := - i \int_0^{n\tau} \Stau(n \tau - s) \Pi_\tau \left(V_j P_{>N_\tau} \psi(s) \right) \rmd s, \quad j = 1, 2. 
	\end{equation}
	
	\subsection{Proof of \texorpdfstring{\cref{prop:E1_pg2} for $p>2$}{Proposition 3.1 for p>2}}
	We first consider \cref{prop:E11_12} for the case $p > 2$ and $d=1, 2$. 
	\begin{proof}[Proof of \cref{prop:E11_12}]
		For $\mathcal{E}_{1, 1}$ in \cref{eq:E1_decomp}, using \cref{lem:Pi}, \cref{eq:V1}, and \cref{lem:dSc} with $(\tilde q, \tilde r) = (\frac{2p}{d}, \frac{2p}{p-2})$, we have
		\begin{align}\label{eq:E1_est}
			\| \mathcal{E}_{1, 1} \|_{l^{q}_\tau(I; L^{r})} 
			&\lesssim \| \Pi_\tau (V_1 P_{> N_\tau} \psi) \|_{L^{\tilde q'}([0, T]; L^{\tilde r'})} \lesssim \| V_1 \|_{L^p} \| P_{>N_\tau} \psi \|_{L^{\tilde q'}([0, T]; L^{2})} \notag \\
			&\lesssim \frac{|T|^{1 - \frac{d}{2p}}}{N_\tau^{2}}  \| \psi \|_{L^\infty([0, T]; H^{2})} \sim \tau.  
		\end{align}
		Similarly, for $\mathcal{E}_{1, 2}$ in \cref{eq:E1_decomp}, using \cref{lem:Pi} and \cref{lem:dSc} with $(\tilde q, \tilde r) = (\infty, 2)$, we have
		\begin{align}\label{eq:E2_est}
			\| \mathcal{E}_{1, 2} \|_{l^{q}_\tau(I; L^{r})} 
			&\lesssim \| \Pi_\tau (V_2 P_{> N_\tau} \psi) \|_{L^{1}([0, T]; L^{2})} \lesssim \| V_2 \|_{L^\infty} \| P_{>N_\tau} \psi \|_{L^{1}([0, T]; L^{2})} \notag \\
			&\lesssim \frac{|T|}{N_\tau^{2}}  \| \psi \|_{L^\infty([0, T]; H^{2})} \sim \tau.  
		\end{align}
		These complete the proof for \cref{prop:E11_12}. 		 
	\end{proof}
	
	Next, we consider the 3D case \cref{prop:E11_3}. The estimate of $\mathcal{E}_{1, 2}$ remains the same as that of \cref{eq:E2_est} in 3D. However, the estimate of $\mathcal{E}_{1, 1}$ is different as the choice of $\tilde r = \frac{2p}{p-2}$ may no longer be admissible in 3D. 
	To estimate $\mathcal{E}_{1, 1}$ in 3D, we define
	\begin{equation}\label{eq:R_def}
		R = R(t) := \int_0^t S_\tau(t-s) \Pi_\tau( V_1 P_{>N} \psi(s)) \rmd s, \qquad N \geq 1. 
	\end{equation}
	We note that $\mathcal{E}_{1,1}^n = R(n\tau)$ in \cref{eq:R_def} with $N = N_\tau$. Hence, it suffices to establish the estimate for $R$, which is in fact independent of the temporal discretization and may be of independent interest. This is also the reason why we do not fix $N = N_\tau$ in the definition of $R$.  
	
	For $R$ \cref{eq:R_def}, we have the following. 
	\begin{proposition}\label{prop:R1}
		When $d=3$ and $p>2$, for any $N \geq 1$, we have
		\begin{equation*}
			\| R \|_{l^{q_0}_\tau{([0, T]; L^{r_0})}} + \| R \|_{l^{\infty}_\tau{([0, T]; L^{2})}} \lesssim \frac{1}{N^2}. 
		\end{equation*}
	\end{proposition}
	
	\begin{proof}
		We make the following specific choice of $\vep$ in \cref{eq:r1_def}. As $p>2$, we may assume that $p$ is arbitrarily close to $2$ such that $p = \frac{2}{1-\theta}$ for some $\theta > 0$ sufficiently small. Then, we choose $\vep = \theta/2$ in \cref{eq:r1_def} such that from \cref{eq:r1_def,eq:r0_def},
		\begin{equation}
			(q_0, r_0) = \left( \frac{4}{1-2\theta}, \frac{3}{1+\theta} \right),  
		\end{equation}
		which can be arbitrarily close to the admissible pair $(4, 3)$ in 3D. 
		
		We also summarize in the following some index relations to be used in the Bernstein inequality \cref{eq:bernstein} and Sobolev embedding: for $p > 2$ and $d=3$, 
		\begin{equation}\label{eq:embed1}
			r_0 \xrightarrow{2-\theta}1, \qquad r_1' \xrightarrow{\frac{1+\theta}{2}} 1, 
		\end{equation}
		where, for any $1 \leq p_1 < p_2 \leq \infty$ and $s>0$, $ p_2 \xrightarrow{s} p_1$ represents the relation
		\begin{equation}\label{eq:idx_relation}
			3(\frac{1}{p_1} - \frac{1}{p_2}) = s.  
		\end{equation}
		
		Note that $R$ \cref{eq:R_def} can be equivalently written as
		\begin{equation}\label{eq:R_rewrite}
			R = \int_0^t S_\tau(t-s) \Pi_\tau ( V_1 |\nabla|^{-2} P_{>N} (-\Delta) \psi(s) ) \rmd s. 
		\end{equation}
		According to the discussion in the introduction, we plug \cref{NLSE} into \cref{eq:R_rewrite} to obtain 
		\begin{equation}\label{eq:R_decomp_1}
			R = R_1 - R_2, 
		\end{equation}
		where
		\begin{align}
			&R_1 = R_1(t) := \int_0^t S_\tau(t-s) \Pi_\tau ( V_1 |\nabla|^{-2} P_{>N} (i\partial_t \psi(s) - V_2 \psi(s)) ) \rmd s, \label{eq:R1_def} \\
			&R_2 = R_2(t) :=	\int_0^t S_\tau(t-s) \Pi_\tau ( V_1 |\nabla|^{-2} P_{>N} (V_1 \psi(s)) ) \rmd s. \label{eq:R2_def} 
		\end{align}
		For $R_1$, using \cref{lem:dSc}, \cref{eq:V1}, \cref{lem:Pi}, Holder's inequality, the space-time estimates of $\partial_t \psi$ \cref{eq:regularity}, and $\psi(t) \in L^2(\R^3) \cap L^\infty(\R^3)$, we have, for $(q, r) \in \{(q_0, r_0), (\infty, 2)\}$, 
		\begin{align}\label{eq:R1}
			\| R_1 \|_{l^{q}_\tau{([0, T]; L^{r})}} 
			&\lesssim \| V_1 |\nabla|^{-2}P_{>N} (i\partial_t \psi - V_2 \psi) \|_{L^{q_1'}{([0, T]; L^{r_1'})}} \notag \\
			&\leq \| V_1 \|_{L^p} (\| \, |\nabla|^{-2}P_{>N} \partial_t \psi \|_{L^{q_1'}{([0, T]; L^{r_0})}} + \| \, |\nabla|^{-2}P_{>N} V_2 \psi \|_{L^{q_1'}{([0, T]; L^{r_0})}}) \notag \\
			&\lesssim\frac{|T|^{1-\frac{1}{q_0} - \frac{1}{q_1}}}{N^{2}} \| \partial_t \psi \|_{L^{q_0}{([0, T]; L^{r_0})}} + \frac{|T|^{1- \frac{1}{q_1}}}{N^{2}} \| V_2 \psi \|_{L^{\infty}{([0, T]; L^{r_0})}} \notag \\
			&\lesssim \frac{1}{N^{2}} + \frac{1}{N^2}\|V_2\|_{L^\infty} \| \psi \|_{L^\infty([0, T]; L^{r_0})} \lesssim \frac{1}{N^2}.  
		\end{align}
		
		It remains to estimate $R_2$. We further decompose $R_2$ as
		\begin{equation}\label{eq:R2_decomp}
			R_2 = R_{2,0} + R_{2,1} + R_{2,2}, 
		\end{equation}
		where, we have, recalling \cref{eq:dyadic_decomposition,eq:PN},  
		\begin{align}
			&R_{2,0} = \int_0^t S_\tau(t-s) \Pi_\tau P_{< 1} ( V_1 |\nabla|^{-2} P_{>N} (V_1 \psi(s)) ) \rmd s, \label{eq:R20_def}\\
			&R_{2,1} = \int_0^t S_\tau(t-s) \Pi_\tau \sum_{M \geq 1} P_M ( V_1 |\nabla|^{-2} P_{>N} (V_1 P_{\gg M} \psi(s)) ) \rmd s, \label{eq:R21_def}\\
			&R_{2,2} = \int_0^t S_\tau(t-s) \Pi_\tau \sum_{M \geq 1} P_M ( V_1 |\nabla|^{-2} P_{>N} (V_1 P_{\lesssim M} \psi(s)) ) \rmd s. \label{eq:R22_def}
		\end{align}
		
		In the following, we estimate $R_{2,0}$, $R_{2,1}$, and $R_{2,2}$, respectively. We start with $R_{2, 0}$. By \cref{lem:dSc,lem:Pi}, for $(q, r) \in \{(q_0, r_0), (\infty, 2)\}$, since $V_1 \in L^1(\R^3) \cap L^{p}(\R^3)$, we have
		\begin{align}\label{eq:R20}
			\| R_{2, 0} \|_{l^{q}_\tau{([0, T]; L^{r})}} 
			&\lesssim \| P_{< 1} ( V_1 |\nabla|^{-2} P_{>N} (V_1 \psi) )  \|_{L^{q_1'}{([0, T]; L^{r_1'})}} \notag \\
			&\lesssim \| V_1 |\nabla|^{-2} P_{>N} (V_1 \psi) \|_{L^{q_1'}{([0, T]; L^1)}} \notag \\
			&\lesssim \| V_1 \|_{L^2} \| \, |\nabla|^{-2} P_{>N} (V_1 \psi) \|_{L^{q_1'}{([0, T]; L^{2})}} \notag \\
			&\lesssim \frac{\| V_1 \|_{L^2}}{N^2}\| V_1 \psi \|_{L^{q_1'}{([0, T]; L^2)}} \notag \\
			&\lesssim |T|^{1-\frac{1}{q_1}} \frac{\| V_1 \|_{L^2}^2}{N^2} \| \psi \|_{L^{\infty}{([0, T]; L^\infty)}} \lesssim \frac{1}{N^2}. 
		\end{align}
		Then we estimate $R_{2, 1}$ where $\psi$ is of high-frequency. We have, for $(q, r) \in \{(q_0, r_0), (\infty, 2)\}$, by \cref{lem:dSc,lem:Pi}, recalling $r_1' \xrightarrow{\frac{1+\theta}{2}} 1$ in \cref{eq:embed1}, 
		\begin{align}\label{eq:R21_1}
			\| R_{2, 1} \|_{l^{q}_\tau{([0, T]; L^{r})}} 
			&\lesssim \| \sum_{M \geq 1} P_M ( V_1 |\nabla|^{-2} P_{>N} (V_1 P_{\gg M} \psi) ) \|_{L^{q_1'}{([0, T]; L^{r_1'})}} \notag \\
			&\lesssim  \sum_{M \geq 1} \| P_M ( V_1 |\nabla|^{-2} P_{>N} (V_1 P_{\gg M} \psi) ) \|_{L^{q_1'}{([0, T]; L^{r_1'})}} \notag \\
			&\lesssim \sum_{M \geq 1} M^{\frac{1}{2}+\frac{\theta}{2}} \| V_1 |\nabla|^{-2} P_{>N} (V_1 P_{\gg M} \psi) \|_{L^{q_1'}{([0, T]; L^{1})}} \notag \\
			&\lesssim \| V_1 \|_{L^p} \sum_{M \geq 1} M^{\frac{1}{2}+\frac{\theta}{2}} \| \, |\nabla|^{-2} P_{>N} (V_1 P_{\gg M} \psi) \|_{L^{q_1'}{([0, T]; L^{p'})}} \notag \\
			&\lesssim \frac{\| V_1 \|_{L^p}}{N^2} \sum_{M \geq 1} M^{\frac{1}{2}+\frac{\theta}{2}} \|V_1 P_{\gg M} \psi\|_{L^{q_1'}{([0, T]; L^{p'})}} \notag \\
			&\lesssim \frac{\| V_1 \|^2_{L^p}}{N^2} \sum_{M \geq 1} M^{\frac{1}{2}+\frac{\theta}{2}} \|P_{\gg M} \psi\|_{L^{q_1'}{([0, T]; L^\frac{p}{p-2})}} \notag \\
			&\lesssim \frac{\| V_1 \|^2_{L^p}}{N^2} \sum_{M \geq 1} M^{\frac{1}{2}+\frac{\theta}{2}} \sum_{M' \gg M} (M')^{\frac{6}{p} - \frac{3}{2}} (M')^{-2} \| \, |\nabla|^{2} P_{M'} \psi\|_{L^{q_1'}{([0, T]; L^2)}} \notag \\
			&\lesssim \frac{|T|^{1-\frac{1}{q_1}}}{N^2}  \| \psi \|_{L^\infty([0, T]; H^2)} \sum_{M \geq 1} \sum_{M' \gg M} M^{\frac{1}{2}+\frac{\theta}{2}}  (M')^{-\frac{1}{2} - 3\theta} \lesssim \frac{1}{N^2}, 
		\end{align}
		where, in the last inequality, recalling that $M$ is taken over dyadic numbers, 
		\begin{equation}\label{eq:sumM}
			\sum_{M \geq 1} \sum_{M' \gg M} M^{\frac{1}{2}+\frac{\theta}{2}}  (M')^{-\frac{1}{2} - 3\theta} \lesssim \sum_{M \geq 1} M^{\frac{1}{2}+\frac{\theta}{2}}  M^{-\frac{1}{2} - 3\theta} = \sum_{M \geq 1} M^{-\frac{5}{2}\theta} < \infty. 
		\end{equation}
		
		Next, we consider $R_{2, 2}$, where $\psi$ is of low frequency, by performing a partial normal form transformation. Denoting
		\begin{equation}
			W(t) = \Pi_\tau \sum_{M \geq 1} P_M ( V_1 |\nabla|^{-2} P_{>N} (V_1 P_{\lesssim M} \psi(t)) ), 
		\end{equation}
		and taking the Fourier transform, we get
		\begin{equation}
			\widehat{R_{2, 2}}(\vxi) = e^{- i t |\vxi|^2} \int_0^t e^{i s |\vxi|^2} \widehat{W}(\vxi, s) \rmd s. 
		\end{equation}
		Note that $\widehat{W}(\vxi, t) = 0$ when $|\vxi| \leq 1/2$. By integration by parts, we have 
		\begin{equation}
			\widehat{R_{2, 2}}(\vxi) = \frac{1}{i|\vxi|^2}\left( \widehat{W}(\vxi, t) - e^{- i t |\vxi|^2} \widehat{W}(\vxi, 0) - \int_0^t e^{- i (t-s)|\vxi|^2} \partial_s \widehat{W}(\vxi, s) \rmd s \right),  
		\end{equation}
		which yields, by the inverse Fourier transform, 
		\begin{align}\label{eq:R22_integrationbyparts}
			iR_{2,2} 
			&= |\nabla|^{-2} \Pi_\tau \sum_{M \geq 1} P_M \left( V_1 |\nabla|^{-2} P_{>N} (V_1 P_{\lesssim M} \psi(t)) \right) \notag \\
			&\quad -  S_\tau(t) |\nabla|^{-2} \Pi_\tau  \sum_{M \geq 1} P_M \left( V_1 |\nabla|^{-2} P_{>N} (V_1 P_{\lesssim M} \psi_0) \right) \notag \\
			&\quad + \int_0^t S_\tau(t-s) |\nabla|^{-2} \Pi_\tau \sum_{M \geq 1} P_M \left( V_1 |\nabla|^{-2} P_{>N} (V_1 P_{\lesssim M} \partial_t \psi(s)) \right) \rmd s \notag \\
			&=: R_{2,2,1} + R_{2,2,2} + R_{2,2,3}. 
		\end{align}
		For $R_{2,2,1}$, we have, by \cref{lem:Pi} and recalling $r_0 \xrightarrow{2-\theta}1$ in \cref{eq:embed1}, 
		\begin{align}\label{eq:R221_1}
			\| R_{2,2,1} \|_{l^{q_0}_\tau{([0, T]; L^{r_0})}}
			&\leq \sum_{M \geq 1} \| P_M |\nabla|^{-2} \Pi_\tau \left( V_1 |\nabla|^{-2} P_{>N} (V_1 P_{\lesssim M} \psi) \right) \|_{l^{q_0}_\tau{([0, T]; L^{r_0})}} \notag \\
			&\lesssim \sum_{M \geq 1} M^{-\theta} \|V_1 (|\nabla|^{-2} P_{>N} (V_1 P_{\lesssim M} \psi)) \|_{l^{q_0}_\tau{([0, T]; L^{1})}} \notag \\
			&\lesssim \| V_1 \|_{L^2} \sum_{M \geq 1} M^{-\theta} \|\,|\nabla|^{-2} P_{>N} (V_1 P_{\lesssim M} \psi) \|_{l^{q_0}_\tau{([0, T]; L^{1})}} \notag \\
			&\lesssim \frac{\| V_1 \|^2_{L^2}}{N^2} \sum_{M \geq 1} M^{-\theta} \| P_{\lesssim M} \psi \|_{l^{q_0}_\tau{([0, T]; L^\infty)}} \notag \\
			& \lesssim \frac{|T|^\frac{1}{q_0}}{N^2} \| \psi \|_{L^\infty{([0, T]; L^\infty)}} \sum_{M \geq 1} M^{-\theta} \lesssim \frac{1}{N^2},  
		\end{align}
		and, similarly, 
		\begin{align}\label{eq:R221_2}
			\| R_{2,2,1} \|_{l^{\infty}_\tau{([0, T]; L^{2})}}
			&\lesssim \sum_{M \geq 1} \| P_M |\nabla|^{-2} \Pi_\tau \left( V_1 |\nabla|^{-2} P_{>N} (V_1 P_{\lesssim M} \psi) \right) \|_{l^{\infty}_\tau{([0, T]; L^{2})}} \notag \\
			&\lesssim \sum_{M \geq 1} M^{-\frac{1}{2}} \|V_1 (|\nabla|^{-2} P_{>N} (V_1 P_{\lesssim M} \psi)) \|_{l^{q_0}_\tau{([0, T]; L^{1})}} \notag \\
			&\lesssim \| V_1 \|_{L^2} \sum_{M \geq 1} M^{-\frac{1}{2}} \| \, |\nabla|^{-2} P_{>N} (V_1 P_{\lesssim M} \psi) \|_{l^{q_0}_\tau{([0, T]; L^{2})}} \notag \\
			&\lesssim \frac{\| V_1 \|^2_{L^2}}{N^2} \sum_{M \geq 1} M^{-\frac{1}{2}} \| P_{\lesssim M} \psi \|_{l^{q_0}_\tau{([0, T]; L^\infty)}} \lesssim \frac{1}{N^2}.  
		\end{align}
		Then we estimate $R_{2,2,2}$. For $(q, r) \in \{(q_0, r_0), (\infty, 2)\}$, using \cref{lem:dS1,lem:Pi}, 
		\begin{align}\label{eq:R222}
			\| R_{2,2,2} \|_{l^{q}_\tau{([0, T]; L^{r})}}
			&\lesssim \| \, |\nabla|^{-2}  \sum_{M \geq 1} P_M ( V_1 |\nabla|^{-2} P_{>N} (V_1 P_{\lesssim M} \psi_0)) \|_{L^2} \notag \\
			&\lesssim \sum_{M \geq 1} \| P_M |\nabla|^{-2} ( V_1 |\nabla|^{-2} P_{>N} (V_1 P_{\lesssim M} \psi_0)) \|_{L^2} \notag \\
			&\lesssim  \sum_{M \geq 1} M^{-\frac{1}{2}} \| V_1 |\nabla|^{-2} P_{>N} (V_1 P_{\lesssim M} \psi_0) \|_{L^1} \notag \\
			&\lesssim \frac{\| V_1 \|^2_{L^2}}{N^2} \sum_{M \geq 1} M^{-\frac{1}{2}} \| P_{\lesssim M} \psi_0 \|_{L^\infty} \notag \\ 
			&\lesssim \frac{1}{N^2} \| \psi_0 \|_{L^\infty} \lesssim \frac{1}{N^2}.   
		\end{align}
		Finally, we consider $R_{2,2,3}$. We have, for $(q, r) \in \{(q_0, r_0), (\infty, 2)\}$, using \cref{lem:dSc,lem:Pi}, recalling $r_1' \xrightarrow{\frac{1+\theta}{2}} 1$ in \cref{eq:embed1}, 
		\begin{align}\label{eq:R223}
			\| R_{2,2,3} \|_{l^{q}_\tau{([0, T]; L^{r})}} 
			&\lesssim \| \, |\nabla|^{-2} \Pi_\tau \sum_{M \geq 1} P_M ( V_1 |\nabla|^{-2} P_{>N} (V_1 P_{\lesssim M} \partial_t \psi) ) \|_{L^{q_1'}{([0, T]; L^{r_1'})}} \notag \\
			&\lesssim  \sum_{M \geq 1} \| P_M |\nabla|^{-2} ( V_1 |\nabla|^{-2} P_{>N} (V_1 P_{\lesssim M} \partial_t \psi) ) \|_{L^{q_1'}{([0, T]; L^{r_1'})}} \notag \\
			&\lesssim  \sum_{M \geq 1} M^{-\frac{3}{2}+\frac{\theta}{2}} \| V_1 |\nabla|^{-2} P_{>N} (V_1 P_{\lesssim M} \partial_t \psi) \|_{L^{q_1'}{([0, T]; L^{1})}} \notag \\
			&\lesssim \| V_1 \|_{L^2} \sum_{M \geq 1} M^{-\frac{3}{2}+\frac{\theta}{2}} \|\, |\nabla|^{-2} P_{>N} (V_1 P_{\lesssim M} \partial_t \psi) \|_{L^{q_1'}{([0, T]; L^{2})}} \notag \\
			&\lesssim \frac{\| V_1 \|_{L^2}^2}{N^2} \sum_{M \geq 1} M^{-\frac{3}{2} + \frac{\theta}{2}} \|  P_{\lesssim M} \partial_t \psi \|_{L^{q_1'}{([0, T]; L^{\infty})}} \notag \\
			&\lesssim \frac{\| V_1 \|_{L^2}^2}{N^2} |T|^{1 - \frac{1}{q_1} - \frac{1}{2}} \sum_{M \geq 1} M^{-\frac{3}{2} +\frac{\theta}{2}} M^{\frac{1}{2}} \|  \partial_t \psi \|_{L^{2}{([0, T]; L^6)}} \lesssim \frac{1}{N^2}. 
		\end{align}
		Combining \cref{eq:R221_1,eq:R221_2,eq:R222,eq:R223}, we obtain the estimate of $R_{2,2}$ by \cref{eq:R22_integrationbyparts}, which, together with \cref{eq:R20,eq:R21_1} concludes the estimate of $R_2$ in \cref{eq:R_decomp_1}. Then the proof is completed. 
	\end{proof}
	
	With \cref{prop:R1}, \cref{prop:E11_3} is a direct corollary, and we omit the proof. The proof of \cref{prop:E1_pg2} is completed. 
	
	\subsection{Proof of \texorpdfstring{\cref{prop:E1_pl2} for $p \leq 2$}{Proposition 3.2 for p <=2}}
	We consider the case $\frac{d}{2} < p \leq 2$ and $p \geq 1$ in this subsection. For $\vep$ in \cref{eq:r1_def}, we simply choose $\vep$ sufficiently small. We start with \cref{prop:E12_12} for the case $d=1, 2$.
	
	\begin{proof}[Proof of \cref{prop:E12_12}]
		Recalling \cref{eq:E1_decomp}, for $\mathcal{E}_{1, 2}$, using \cref{lem:Pi}, \cref{lem:dSc} with $(\tilde q, \tilde r) = (\infty, 2)$, for any admissible pairs $(q, r)$ with $q \neq 2$, we have
		\begin{align}\label{E12}
			\| \mathcal{E}_{1, 2} \|_{l^{q}_\tau(I; L^{r})} 
			&\lesssim \| \Pi_\tau V_2 P_{>N_\tau} \psi \|_{L^{1}([0, T]; L^{2})} \notag \\
			&\lesssim \| V_2 \|_{L^\infty} \| P_{>N_\tau} \psi \|_{L^{1}([0, T]; L^{2})} \notag \\
			&\lesssim \frac{|T|}{N_\tau^{2-\alpha}} \| \, |\nabla|^{2-\alpha} \psi \|_{L^\infty([0, T]; L^{2})} \sim \tau^{1 - \frac{\alpha}{2}}. 
		\end{align}
		Then we estimate $\mathcal{E}_{1, 1}$. We consider the case $d=1$ and $d=2$ separately. When $d=1$, for any admissible pair $(q, r)$ with $q \neq 2$, using \cref{lem:Pi} and \cref{lem:dSc} with $(\tilde q, \tilde r) = (q_1, r_1)$, recalling \cref{eq:r1_def,eq:V1}, we have, by the Sobolev embedding $H^\alpha \hookrightarrow L^{r_0}$ (recalling that $\alpha = \frac{1}{2}^+$ if $r_0 = \infty$),
		\begin{align}\label{E11_1D}
			\| \mathcal{E}_{1, 1} \|_{l^{q}_\tau(I; L^{r})} 
			&\lesssim \| \Pi_\tau (V_1 P_{>N_\tau} \psi ) \|_{L^{q_1'}([0, T]; L^{r_1'})} \notag \\
			&\lesssim \| V_1 \|_{L^p} \| P_{>N_\tau} \psi \|_{L^{q_1'}([0, T]; L^{r_0})} \notag \\
			&\lesssim \| P_{>N_\tau} |\nabla|^\alpha \psi \|_{L^{q_1'}([0, T]; L^{2})} \notag \\
			&\lesssim \frac{|T|^{1 - \frac{1}{q_1}}}{N_\tau^{2- 2\alpha}}\| \psi \|_{L^\infty([0, T]; H^{2-\alpha})} \sim \tau^{1-\alpha}.
		\end{align}
		When $d=2$, similar to \cref{E11_1D}, for any admissible pair $(q, r)$ with $q \neq 2$, using \cref{lem:dSc,lem:Pi}, recalling the Sobolev embedding $H^{2(\frac{1}{2} - \frac{1}{r_0})} \hookrightarrow L^{r_0}$ and verifying that $2(\frac{1}{2} - \frac{1}{r_0}) = \alpha + \frac{2}{r_1}$ when $p\leq2$, we have
		\begin{align}\label{E11_2D}
			\| \mathcal{E}_{1, 1} \|_{l^{q}_\tau(I; L^{r})} 
			&\lesssim \| \Pi_\tau \left(V_1 P_{>N_\tau} \psi \right) \|_{L^{q_1'}([0, T]; L^{r_1'})} \notag \\
			&\lesssim \| V_1 \|_{L^p} \| P_{>N_\tau} \psi \|_{L^{q_1'}([0, T]; L^{r_0})} \notag \\
			&\lesssim \| P_{>N_\tau} |\nabla|^{\alpha + \vep} \psi \|_{L^{q_1'}([0, T]; L^{2})} \lesssim \frac{|T|^{1 - \frac{1}{q_1}}}{N_\tau^{2-2\alpha - \vep}}  \| \psi \|_{L^\infty([0, T]; H^{2-\alpha})} \sim \tau^{1^--\alpha}, && 1 < p \leq 2. 
		\end{align}
		Combining \cref{E11_1D,E11_2D,E12} proves \cref{prop:E12_12}.
	\end{proof} 
	
	Then we consider the 3D case \cref{prop:E12_3}. The estimate of $\mathcal{E}_{1, 2}$ remains the same as \cref{E12}. However, the estimate of $\mathcal{E}_{1, 1}$ requires a refined analysis. Similar to \cref{prop:R1}, we establish a general result in terms of $R$ \cref{eq:R_def}, which may be of independent interest. 
	
	In fact, the same arguments of using discrete Strichartz estimates and Sobolev embeddings as in \cref{E11_1D,E11_2D} can be carried over to the 3D case. However, this can only give a suboptimal error bound of order $O(N_\tau^{-(\frac{3}{2} - 2\alpha)})$ (up to some $\vep$ loss of order). In comparison, with a more refined estimate of using the normal form transformation and the high-low frequency decomposition as in the proof of \cref{prop:R1}, we are able to improve such estimate to $O(N_\tau^{-(2 - 3\alpha)})$, hence, gaining $(\frac{1}{2} - \alpha)$-order.
	
	\begin{proposition}\label{prop:R2}
		When $d=3$ and $\frac{3}{2} < d \leq 2$, for any $N \geq 1$, we have
		\begin{equation*}
			\| R \|_{l^{q_0}_\tau{([0, T]; L^{r_0})}} + \| R \|_{l^{\infty}_\tau{([0, T]; L^{2})}} \lesssim \frac{1}{N^{2^- - 3\alpha}}, \qquad  \frac{3}{2} < p \leq 2. 
		\end{equation*}
	\end{proposition}
	
	To prove \cref{prop:R2}, we will use the following index relations
	\begin{equation}\label{eq:embed}
		r_0 \xrightarrow{\alpha+\frac{1}{2}+\vep}2\xrightarrow{1-\vep}r_1'\xrightarrow{\frac{1}{2}+\vep}1, \qquad p' \xrightarrow{\alpha} 2 \xrightarrow{\alpha} p,  
	\end{equation}
	where we recall \cref{eq:idx_relation} for the definition. Moreover, we need the following two auxiliary results.  
	\begin{lemma}\label{lem:move_derivative}
		For any $M \geq 1$ and $0 \leq \gamma \leq \frac{1}{2}$, we have
		\begin{equation*}
			\| P_{M} (V_1 \phi) \|_{L^{r_1'}} \lesssim M^{\frac{1}{2}+ \vep - \gamma} \| V_1 \|_{L^p} \| |\nabla|^{\gamma} \phi \|_{L^{p'}}. 
		\end{equation*}
	\end{lemma}
	
	\begin{proof}
		Using \cref{lem:Pi}, we obtain
		\begin{equation}
			\| P_{M} (V_1 \phi) \|_{L^{r_1'}} \lesssim M^{\frac{1}{2}+\vep - \gamma} \| V_1 \phi \|_{L^{q_1}},  
		\end{equation}
		where $1 \leq q_1 \leq r_1'$ satisfies
		\begin{equation*}
			\frac{1}{2}+\vep - \gamma = 3(\frac{1}{q_1} - \frac{1}{r_1'}). 
		\end{equation*}
		Furthermore, using Holder's inequality and Sobolev embedding, we have
		\begin{equation}
			\| V_1 \phi \|_{L^{q_1}} \leq \| V_1 \|_{L^p} \| \phi \|_{L^{q_2}} \lesssim \| V_1 \|_{L^p} \| |\nabla|^{\gamma} \phi \|_{L^{p'}},  
		\end{equation}
		where $q_1 \leq q_2 <\infty$ satisfies $\frac{1}{q_1} = \frac{1}{p} + \frac{1}{q_2}$ and we verify that
		\begin{equation*}
			3(\frac{1}{p'} - \frac{1}{q_2}) = 3(1 - \frac{1}{q_1}) = \gamma, 
		\end{equation*}
		which concludes the proof. 
	\end{proof}
	
	\begin{remark}
		By the Sobolev embedding, recalling \cref{eq:embed}, we have
		\begin{equation*}
			\| P_{M} (V_1 \phi) \|_{L^{r_1'}} \lesssim \| |\nabla|^{\frac{1}{2} + \vep} P_{M} (V_1 \phi) \|_{L^1} = \| |\nabla|^{\frac{1}{2} + \vep - \gamma} P_{M} |\nabla|^\gamma (V_1 \phi) \|_{L^1}. 
		\end{equation*}
		Hence, though $|\nabla|^\gamma$ and $V_1$ do not commute, \cref{lem:move_derivative} may be viewed as an $L^p$-based mechanism for transferring the derivative $|\nabla|^\gamma$ from the product $V_1\phi$ onto $\phi$.
	\end{remark}
	
	In the same spirit of \cref{lem:move_derivative}, we have the following. 
	\begin{lemma}\label{lem:move_derivative2}
		For any $M \geq 1$ and $\delta>0$ sufficiently small, we have
		\begin{equation*}
			\| P_{M} |\nabla|^{-2} (V_1 \phi) \|_{L^{r_0}} \lesssim M^{-\delta} \| V_1 \|_{L^p} \| |\nabla|^{\alpha+\vep+\delta} \phi \|_{L^{p'}}. 
		\end{equation*}
	\end{lemma}
	
	\begin{proof}
		Using \cref{lem:Pi}, we have
		\begin{equation}
			\| P_{M} |\nabla|^{-2} (V_1 \phi) \|_{L^{r_0}} \lesssim M^{-\delta} \| V_1 \phi \|_{L^{q_1}}, 
		\end{equation}
		where $ 1 < q_1 \leq r_0$ satisfies
		\begin{equation*}
			3(\frac{1}{q_1} - \frac{1}{r_0}) = 2 - \delta. 
		\end{equation*}
		Furthermore, using H\"older's inequality and Sobolev embedding, we have
		\begin{equation}
			\| V_1 \phi \|_{L^{q_1}} \leq \| V_1 \|_{L^p} \| \phi \|_{L^{q_2}} \lesssim \| V_1 \|_{L^p} \| |\nabla|^{\alpha+\vep+\delta} \phi \|_{L^{p'}},  
		\end{equation}
		where $q_1 \leq q_2 <\infty$ satisfies $\frac{1}{q_1} = \frac{1}{p} + \frac{1}{q_2}$ and we verify that, recalling $r_0 \xrightarrow{\alpha+2+\vep} 1$ in \cref{eq:embed}, 
		\begin{equation*}
			3(\frac{1}{p'} - \frac{1}{q_2}) = 3(1 - \frac{1}{q_1}) = 3(1 - \frac{1}{r_0}) - (2 - \delta) = \alpha + \vep + \delta, 
		\end{equation*}
		which concludes the proof. 
	\end{proof}

	\begin{proof}[Proof of \cref{prop:R2}]
		We have the same decomposition of $R = R_1 - R_2$ as in \cref{eq:R_decomp_1}. For $R_1$ \cref{eq:R1_def}, similar to \cref{eq:R1}, we have, for $(q, r) \in \{(q_0, r_0), (\infty, 2)\}$, 
		\begin{align}\label{eq:R1_est_2}
			\| R_1 \|_{l^{q}_\tau{([0, T]; L^{r})}} 
			&\lesssim\frac{|T|^{1-\frac{1}{q_0} - \frac{1}{q_1}}}{N^{2-\alpha}} \| \langle\nabla\rangle^{-\alpha} \partial_t \psi \|_{L^{q_0}{([0, T]; L^{r_0})}} + \frac{|T|^{1- \frac{1}{q_1}}}{N^{2}} \| V_2 \psi \|_{L^{\infty}{([0, T]; L^{r_0})}} \notag \\
			&\lesssim \frac{1}{N^{2-\alpha}} + \frac{1}{N^2}\|V_2\|_{L^\infty} \| \psi \|_{L^\infty([0, T]; L^{r_0})} \lesssim \frac{1}{N^{2-\alpha}}.  
		\end{align}
		
		For $R_2$ \cref{eq:R2_def}, we recall the decomposition $R_2 = R_{2, 0} + R_{2, 1} + R_{2, 2}$ in \cref{eq:R2_decomp}. 
		For $R_{2, 0}$ \cref{eq:R20_def}, similar to \cref{eq:R20}, using \cref{lem:dSc} and the Sobolev embedding with $p' \xrightarrow{2\alpha} p$ in \cref{eq:embed}, we get, for $(q, r) \in \{(q_0, r_0), (\infty, 2)\}$, 
		\begin{align}\label{eq:R20_2}
			\| R_{2, 0} \|_{l^{q}_\tau{([0, T]; L^{r})}}
			&\lesssim \| V_1 \|_{L^p} \| \, |\nabla|^{-2} P_{>N} (V_1 \psi) \|_{L^{q_1'}{([0, T]; L^{p'})}} \notag \\
			&\lesssim \| V_1 \|_{L^p} \| \, |\nabla|^{-2+2\alpha} P_{>N} (V_1 \psi) \|_{L^{q_1'}{([0, T]; L^{p})}} \notag \\
			&\lesssim |T|^{1-\frac{1}{q_1}} \frac{\| V_1 \|_{L^p}^2}{N^{2-2\alpha}} \| \psi \|_{L^{\infty}{([0, T]; L^\infty)}} \lesssim \frac{1}{N^{2-2\alpha}}.
		\end{align}		
		For $R_{2, 1}$ \cref{eq:R21_def}, using \cref{lem:dSc}, \cref{lem:move_derivative} with $\gamma = \alpha + \vep + \delta$ for some $\delta>0$ sufficiently small and Sobolev embedding with $p' \xrightarrow{2\alpha} p$ in \cref{eq:embed}, we have, for $(q, r) \in \{(q_0, r_0), (\infty, 2)\}$, 
		\begin{align}\label{eq:R21_2}
			\| R_{2, 1} \|_{l^{q}_\tau{([0, T]; L^{r})}} 
			&\lesssim \| \sum_{M \geq 1} P_M ( V_1 |\nabla|^{-2} P_{>N} (V_1 P_{\gg M} \psi) ) \|_{L^{q_1'}{([0, T]; L^{r_1'})}} \notag \\
			&\lesssim \sum_{M \geq 1} \| P_M ( V_1 |\nabla|^{-2} P_{>N} (V_1 P_{\gg M} \psi) ) \|_{L^{q_1'}{([0, T]; L^{r_1'})}} \notag \\
			&\lesssim \sum_{M \geq 1} M^{\frac{1}{2}-\alpha-\delta} \| V_1 \|_{L^p} \| |\nabla|^{-2+\alpha+\vep+\delta} P_{>N} (V_1 P_{\gg M} \psi) \|_{L^{q_1'}{([0, T]; L^{p'})}} \notag \\
			&\lesssim \sum_{M \geq 1} M^{\frac{1}{2}-\alpha-\delta}  \| |\nabla|^{-2+3\alpha+\vep+\delta} P_{>N} (V_1 P_{\gg M} \psi) \|_{L^{q_1'}{([0, T]; L^{p})}} \notag \\
			&\lesssim \frac{1}{N^{2 - 3\alpha -\vep -\delta}} \sum_{M \geq 1} M^{\frac{1}{2}-\alpha-\delta}  \| V_1 P_{\gg M} \psi \|_{L^{q_1'}{([0, T]; L^{p})}} \notag \\
			&\lesssim \frac{|T|^{1 - \frac{1}{q_1}}\| V_1 \|_{L^p}}{N^{2 - 3\alpha -\vep -\delta}} \sum_{M \geq 1}  M^{\frac{1}{2}-\alpha-\delta}  \| P_{\gg M} \psi\|_{L^{\infty}{([0, T]; L^{\infty})}} \notag \\
			&\lesssim \frac{|T|^{1 - \frac{1}{q_1}}\| V_1 \|_{L^p}}{N^{2 - 3\alpha -\vep -\delta}} \sum_{M \geq 1}  M^{\frac{1}{2}-\alpha-\delta} \sum_{M' \gg M} {(M')}^{\frac{3}{2}-2+\alpha} \| \psi \|_{L^{\infty}{([0, T]; H^{2-\alpha})}} \notag \\
			&\lesssim \frac{1}{N^{2-3\alpha - \vep - \delta}}.  
		\end{align}
		To estimate $R_{2, 2}$ \cref{eq:R22_def}, we perform again a partial normal form transformation to obtain $i R_{2, 2} = R_{2, 2, 1} + R_{2, 2, 2} + R_{2, 2, 3}$ as in \cref{eq:R22_integrationbyparts}. For $R_{2, 2, 1} $ in \cref{eq:R22_integrationbyparts}, using \cref{lem:dSc,lem:move_derivative2} and the Sobolev embedding with $p' \xrightarrow{2\alpha} p$ in \cref{eq:embed}, for some $\delta > 0$ sufficiently small, we have, 
		\begin{align}\label{eq:R221_2_1}
			\| R_{2,2,1} \|_{l^{q_0}_\tau{([0, T]; L^{r_0})}}
			&\lesssim  \sum_{M \geq 1} \| P_M |\nabla|^{-2}  ( V_1 |\nabla|^{-2} P_{>N} (V_1 P_{\lesssim M} \psi)) \|_{l^{q_0}_\tau{([0, T]; L^{r_0})}} \notag \\
			&\lesssim \| V_1 \|_{L^p} \sum_{M \geq 1} M^{-\delta}  \|  |\nabla|^{-2+\alpha+\vep+\delta} P_{>N} (V_1 P_{\lesssim M} \psi) \|_{l^{q_0}_\tau{([0, T]; L^{p'})}} \notag \\
			&\lesssim \sum_{M \geq 1} M^{-\delta}  \|  |\nabla|^{-2+3\alpha+\vep+\delta} P_{>N} (V_1 P_{\lesssim M} \psi) \|_{l^{q_0}_\tau{([0, T]; L^{p})}} \notag \\  
			&\lesssim \frac{1}{N^{2-3\alpha-\vep-\delta}} \sum_{M \geq 1} M^{-\delta}  \|  V_1 P_{\lesssim M} \psi \|_{l^{q_0}_\tau{([0, T]; L^{p})}} \notag \\
			&\lesssim \frac{\| V_1 \|_{L^p} |T|^{q_0}}{N^{2-3\alpha-\vep-\delta}} \sum_{M \geq 1} M^{-\delta}   \| \psi \|_{l^{\infty}_\tau{([0, T]; L^{\infty})}} \lesssim \frac{1}{N^{2-3\alpha-\vep-\delta}},  
		\end{align}
		and, by \cref{lem:Pi} and the same Sobolev embedding, 
		\begin{align}\label{eq:R221_2_2}
			\| R_{2,2,1} \|_{l^{\infty}_\tau{([0, T]; L^{2})}}
			&\lesssim  \sum_{M \geq 1} \| P_M |\nabla|^{-2}  ( V_1 |\nabla|^{-2} P_{>N} (V_1 P_{\lesssim M} \psi)) \|_{l^{\infty}_\tau{([0, T]; L^{2})}} \notag \\
			&\lesssim  \sum_{M \geq 1} M^{-\frac{1}{2}} \| V_1 |\nabla|^{-2} P_{>N} (V_1 P_{\lesssim M} \psi) \|_{l^{\infty}_\tau{([0, T]; L^{1})}} \notag \\
			&\lesssim  \| V_1 \|_{L^p} \sum_{M \geq 1} M^{-\frac{1}{2}} \| |\nabla|^{-2} P_{>N} (V_1 P_{\lesssim M} \psi) \|_{l^{\infty}_\tau{([0, T]; L^{p'})}} \notag \\
			&\lesssim \sum_{M \geq 1} M^{-\frac{1}{2}} \| |\nabla|^{-2+2\alpha} P_{>N} (V_1 P_{\lesssim M} \psi) \|_{l^{\infty}_\tau{([0, T]; L^{p})}} \notag \\
			&\lesssim\frac{1}{N^{2-2\alpha}} \sum_{M \geq 1} M^{-\frac{1}{2}}  \| V_1 P_{\lesssim M} \psi \|_{l^{q_0}_\tau{([0, T]; L^{p})}} \notag \\  
			&\lesssim\frac{\| V_1 \|_{L^p}|T|^{q_0}}{N^{2-2\alpha}} \sum_{M \geq 1} M^{-\frac{1}{2}} \| \psi \|_{l^{\infty}_\tau{([0, T]; L^{\infty})}} \lesssim \frac{1}{N^{2-2\alpha}}. 
		\end{align}
		For $R_{2, 2, 2}$ in \cref{eq:R22_integrationbyparts}, similar to \cref{eq:R222}, we have, using the Sobolev embedding with $p' \xrightarrow{2\alpha} p$ in \cref{eq:embed}, for $(q, r) \in \{(q_0, r_0), (\infty, 2)\}$, 
		\begin{align}\label{eq:R222_2}
			\| R_{2,2,2} \|_{l^{q}_\tau{([0, T]; L^{r})}}
			&\lesssim \sum_{M \geq 1} M^{-\frac{1}{2}} \| V_1 |\nabla|^{-2} P_{>N} (V_1 P_{\lesssim M} \psi_0) \|_{L^1} \notag \\
			&\lesssim \| V_1 \|_{L^p} \sum_{M \geq 1} M^{-\frac{1}{2}} \|  |\nabla|^{-2} P_{>N} (V_1 P_{\lesssim M} \psi_0) \|_{L^{p'}} \notag \\
			&\lesssim \sum_{M \geq 1} M^{-\frac{1}{2}} \|  |\nabla|^{-2+2\alpha} P_{>N} (V_1 P_{\lesssim M} \psi_0) \|_{L^{p}} \notag \\
			&\lesssim \frac{1}{N^{2-2\alpha}} \sum_{M \geq 1} M^{-\frac{1}{2}} \| V_1 P_{\lesssim M} \psi_0 \|_{L^{p}} \notag \\
			&\lesssim \frac{\| V_1 \|_{L^p}}{N^{2-2\alpha}} \sum_{M \geq 1} M^{-\frac{1}{2}} \| \psi_0 \|_{L^{\infty}} \lesssim \frac{1}{N^{2-2\alpha}}. 
		\end{align}
		Finally, for $R_{2,2,3}$ in \cref{eq:R22_integrationbyparts}, using \cref{lem:dSc}, \cref{lem:Pi} and the Sobolev embedding, recalling $r_1' \xrightarrow{\frac{1}{2}+\vep} 1$ and $p' \xrightarrow{2\alpha} p$ in \cref{eq:embed}, we have, for $(q, r) \in \{(q_0, r_0), (\infty, 2)\}$, 
		\begin{align}\label{eq:R223_2}
			\| R_{2,2,3} \|_{l^{q}_\tau{([0, T]; L^{r})}} 
			&\lesssim  \sum_{M \geq 1} \| P_M |\nabla|^{-2} ( V_1 |\nabla|^{-2} P_{>N} (V_1 P_{\lesssim M} \partial_t \psi) ) \|_{L^{q_1'}{([0, T]; L^{r_1'})}} \notag \\
			&\lesssim  \sum_{M \geq 1} M^{-\frac{3}{2} + \vep} \| V_1 |\nabla|^{-2} P_{>N} (V_1 P_{\lesssim M} \partial_t \psi) \|_{L^{q_1'}{([0, T]; L^{1})}} \notag \\
			&\lesssim \| V_1 \|_{L^p} \sum_{M \geq 1} M^{-\frac{3}{2}+\vep} \| \, |\nabla|^{-2} P_{>N} (V_1 P_{\lesssim M} \partial_t \psi) \|_{L^{q_1'}{([0, T]; L^{p'})}} \notag \\
			&\lesssim \| V_1 \|_{L^p} \sum_{M \geq 1} M^{-\frac{3}{2}+\vep} \| \, |\nabla|^{-2+2\alpha} P_{>N} (V_1 P_{\lesssim M} \partial_t \psi) \|_{L^{q_1'}{([0, T]; L^{p})}}  \notag \\
			&\lesssim\frac{ \| V_1 \|_{L^p}}{N^{2-2\alpha}} \sum_{M \geq 1} M^{-\frac{3}{2}+\vep} \| V_1 P_{\lesssim M} \partial_t \psi \|_{L^{q_1'}{([0, T]; L^{p})}} \notag \\
			&\lesssim\frac{ \| V_1 \|^2_{L^p}}{N^{2-2\alpha}} \sum_{M \geq 1} M^{-\frac{3}{2}+\vep} \| P_{\lesssim M} \partial_t \psi \|_{L^{q_1'}{([0, T]; L^{\infty})}} \notag \\
			&\lesssim\frac{ \| V_1 \|^2_{L^p}}{N^{2-2\alpha}} |T|^{1 - \frac{1}{q_1} - \frac{1}{2}} \sum_{M \geq 1} M^{-\frac{3}{2}+\vep} M^{\frac{1}{2}+\alpha} \| \langle \nabla \rangle^{-\alpha} \partial_t \psi \|_{L^{2}{([0, T]; L^{6})}} \lesssim \frac{1}{N^{2-2\alpha}}. 
		\end{align}
		Combing \cref{eq:R221_2_1,eq:R221_2_2,eq:R222_2,eq:R223_2}, we obtain the estimate of $R_{2,2}$, which combined with \cref{eq:R20_2,eq:R21_2} yield the error bound for $R_2$ of order $O(N^{-(2^--3\alpha)})$. The conclusion then follows from the estimates of $R_2$ and $R_1$ in \cref{eq:R1_est_2} immediately.  
	\end{proof}
	
	With \cref{prop:R2}, \cref{prop:E12_3} can be obtained by taking $t = t_n$ and $N = N_\tau$ in \cref{eq:R_def}. Then the proof of \cref{prop:E1_pl2} is completed. 
	
	\section{Error estimates of the EWI for \texorpdfstring{\cref{NLSE} with $\beta = 0$}{NLSE with beta=0}}\label{sec:4}
	In this section, we establish the estimates of $\mathcal{E}_0$, $\mathcal{E}_2$, and $Z$ in the error equation \cref{eq:error_eq}, and complete the error estimate for the EWI \cref{eq:EWI_scheme}. 

	\subsection{Temporal error and stability}
	For $\mathcal{E}_0$, a direct application of \cref{lem:dS1} yields that for any admissible pair $(q, r)$ and interval $I \subset \R$, 
	\begin{equation}\label{eq:E0_est}
		\| \mathcal{E}_0 \|_{l^q_\tau(I; L^r)} \leq \| \mathcal{E}_0 \|_{l^q_\tau(\tau \Z; L^r)} \leq \widetilde C_0 \| \Pi_\tau \psi_0 - \psi_\tau^0 \|_{L^2}, 
	\end{equation}
	where $\widetilde C_0$ depends exclusively on $d$ and $q$. Although $\Pi_\tau \psi_0 - \psi_\tau^0 = 0$ in the above, we keep this term since it will be used in extending the local error estimates on some short interval to the global error estimate on $[0, T]$. In that case, $\Pi_\tau \psi_0 - \psi_\tau^0$ will be replaced by $\Pi_\tau \psi_k - \psi_\tau^k$ for some $k$.
	
	Then we consider $\mathcal{E}_2$ \cref{eq:E2}, and we have the following optimal estimates.
	\begin{proposition}\label{prop:E2}
		For any admissible pairs $(q, r)$ with $q \neq 2$, we have
		\begin{equation*}
			\| \mathcal{E}_2 \|_{l_\tau^q([0, T]; L^r)} \lesssim \tau^{1-\frac{\alpha}{2}}. 
		\end{equation*}
	\end{proposition}
	
	\begin{proof}
		First, defining, for $0 \leq k \leq n-1$,  
		\begin{equation}
			w_j(s) = \Pi_\tau [V_j \Pi_\tau (\psi(s) - \psi(k \tau))], \quad  k\tau< s \leq (k+1)\tau, \quad j=1, 2, 
		\end{equation}
		and setting $w_j(s) = 0 $ when $s > T$, we have \cite{bao2025_singular}
		\begin{equation}
			\mathcal{E}_{2}^n = \mathcal{E}_{2, 1}^n + \mathcal{E}_{2, 2}^n := -i \sum_{k=0}^{n-1} \int_{k\tau}^{(k+1)\tau} \Stau(n\tau - s) (w_1(s) + w_2(s)) \rmd s, \quad 1 \leq n \leq T/\tau. 
		\end{equation}
		For any admissible pairs $(q, r)$ and $(\tilde {q}, \tilde{r})$ with $(q, \tilde{q}) \neq (2, 2)$, by \cref{lem:dSc}, we have
		\begin{equation}
			\| \mathcal{E}_{2, j} \|_{l^q_\tau([0, T]; L^r)} \leq \| \mathcal{E}_{2, j} \|_{l_\tau^q(\tau \mathbb{Z}; L^r)} \lesssim \| w_j \|_{L^{\tilde q'}([0, T]; L^{\tilde r'})}, \quad j = 1, 2. 
		\end{equation}
		For $j=1$, choosing $(\tilde {q}, \tilde{r}) = (q_1, r_1)$, we have, with $K := \lfloor T/\tau \rfloor$, 
		\begin{align}\label{eq:w1_est}
			\| \mathcal{E}_{2, 1} \|_{l^q_\tau([0, T]; L^r)} 
			&\lesssim \| w_1 \|_{L^{q_1'}([0, T]; L^{r_1'})} = \left( \int_0^{T} \| w_1(s) \|_{L^{r_1'}}^{q_1'} \rmd s \right)^{\frac{1}{q_1'}} \notag \\
			&= \left( \sum_{k=0}^{K-1} \int_{k\tau}^{(k+1)\tau} \| \Pi_\tau [V_1 \Pi_\tau (\psi(s) - \psi(k \tau))] \|_{L^{r_1'}}^{q_1'} \rmd s \right)^{\frac{1}{q_1'}} \notag \\
			&\lesssim \| V_1 \|_{L^{p}} \left( \sum_{k=0}^{K-1} \int_{k\tau}^{(k+1)\tau} \| \Pi_\tau(\psi(s) - \psi(k \tau)) \|_{L^{ r_0}}^{q_1'} \rmd s \right)^{\frac{1}{q_1'}} \notag \\
			&= \| V_1 \|_{L^{p}} \left( \sum_{k=0}^{K-1} \int_{k\tau}^{(k+1)\tau} \left \| \int_{k\tau}^s \langle \nabla \rangle^{\alpha} \Pi_\tau \langle \nabla \rangle^{-\alpha}\partial_t \psi(t) \rmd t \right \|_{L^{ r_0}}^{q_1'} \rmd s \right)^{\frac{1}{q_1'}} \notag \\
			&\lesssim \tau^{-\frac{\alpha}{2}} \| V_1 \|_{L^{p}} \left( \sum_{k=0}^{K-1} \int_{k\tau}^{(k+1)\tau} \left ( \int_{k\tau}^s \| \langle \nabla \rangle^{-\alpha} \partial_t \psi(t) \|_{L^{r_0}} \rmd t \right )^{q_1'} \rmd s \right)^{\frac{1}{q_1'}} \notag \\
			&\leq \tau^{1-\frac{\alpha}{2}} T^{1 - \frac{1}{q_0} - \frac{1}{q_1}} \| V_1 \|_{L^{p}} \| \, \langle \nabla \rangle^{-\alpha} \partial_t \psi \|_{L^{q_0}([0, T]; L^{r_0})} \lesssim \tau^{1-\frac{\alpha}{2}}. 
		\end{align}
		Similarly, for $j=2$, choosing $(\tilde {q}, \tilde{r}) = (\infty, 2)$, we have
		\begin{align}\label{eq:w2_est}
			\| \mathcal{E}_{2, 2} \|_{l^q_\tau([0, T]; L^r)} 
			&\lesssim \| w_2 \|_{L^{1}([0, T]; L^{2})} = \int_0^{T} \| w_2(s) \|_{L^{2}} \rmd s \notag \\
			&= \sum_{k=0}^{K-1} \int_{k\tau}^{(k+1)\tau} \| \Pi_\tau [V_2 \Pi_\tau (\psi(s) - \psi(k \tau))] \|_{L^{2}} \rmd s \notag \\
			&\lesssim \| V_2 \|_{L^{\infty}}  \sum_{k=0}^{K-1} \int_{k\tau}^{(k+1)\tau} \| \Pi_\tau(\psi(s) - \psi(k \tau)) \|_{L^{2}} \rmd s \notag \\
			&= \| V_2 \|_{L^{\infty}} \int_{k\tau}^{(k+1)\tau} \left \| \int_{k\tau}^s \langle \nabla \rangle^{\alpha} \Pi_\tau \langle \nabla \rangle^{-\alpha}\partial_t \psi(t) \rmd t \right \|_{L^{2}} \rmd s  \notag \\
			&\lesssim \tau^{-\frac{\alpha}{2}} \| V_2 \|_{L^{\infty}}  \sum_{k=0}^{K-1} \int_{k\tau}^{(k+1)\tau} \int_{k\tau}^s \| \langle \nabla \rangle^{-\alpha} \partial_t \psi(t) \|_{L^{2}} \rmd t \rmd s \notag \\
			&\leq \tau^{1-\frac{\alpha}{2}} T \| V_2 \|_{L^{\infty}} \| \langle \nabla \rangle^{-\alpha} \partial_t \psi \|_{L^{\infty}([0, T]; L^{2})} \lesssim \tau^{1-\frac{\alpha}{2}}. 
		\end{align}
		Combining \cref{eq:w1_est,eq:w2_est}, we complete the proof. 
	\end{proof}
	
	\begin{remark}
		In \cref{prop:E2}, the error bound for the temporal discretization is of order $O(\tau^{1 - \frac{\alpha}{2}})$, which is optimal for the EWI with respect to the sharp $H^{2-\alpha}$-regularity of the exact solution. 
	\end{remark}
	
	Then we estimate $Z$ \cref{eq:Z_def}. By \cite[Lemma 11.1]{LRI_error_1}, we have the following. 
	\begin{lemma}\label{lem:phi_1_Pi}
		Let $0 < \tau \leq 1$. For any $ 1 \leq p \leq \infty$ and $f \in L^p(\R^d)$, we have 
		\begin{equation*}
			\| \vphi_1(i \tau \Delta) \Pi_\tau f \|_{L^p} \lesssim \| f \|_{L^p}. 
		\end{equation*}
	\end{lemma}
	
	\begin{proposition}\label{prop:Z}
		Let $I=[0, n\tau]$ for some $0 \leq n \leq T/\tau$. For any admissible pair $(q, r)$ with $q \neq 2$, we have
		\begin{align*}
			\| Z \|_{l_\tau^q(I; L^r)} \leq C_0 (|I|^{1 - \frac{1}{q_1} - \frac{1}{q_0}} + |I|) \| \Pi_\tau \psi - \psi_\tau \|_{X_\tau(I)}, 
		\end{align*}
		where $\| \cdot \|_{X_{\tau}(I)} := \| \cdot \|_{l_\tau^{q_0}(I; L^{r_0})} + \| \cdot \|_{l_\tau^\infty(I; L^2)}$, and $C_0$ is a constant depending exclusively on $\| V_1 \|_{L^p}$ and $\| V_2 \|_{L^\infty}$. 
	\end{proposition}
	
	\begin{proof}
		We have
		\begin{equation}
			Z^n = Z^n_1 + Z^n_2, \qquad Z^n_j =  - i \tau \sum_{k=0}^{n-1} \Stau((n-1)\tau - k \tau) \vphi_1(i \tau \Delta) \Pi_\tau (V_j e^k), \quad j=1,2.  
		\end{equation}
		From \cref{eq:Z_def}, using \cref{lem:dS2} with $(\tilde q, \tilde r) = (q_0, r_0)$, \cref{lem:phi_1_Pi}, \cref{lem:Pi}, and H\"older's inequality, we have
		\begin{align}\label{eq:Z1_est}
			\| Z_1 \|_{l^q_\tau(I; L^r)} 
			&\lesssim \| \vphi_1(i \tau \Delta) \Pi_\tau ( V_1 \Pi_\tau \psi(k\tau) - V_1 \psi^k_\tau) \|_{l_\tau^{q_1'}([0, (n-1)\tau]; L^{ r_1'})} \notag \\
			&\lesssim \| V_1 \|_{L^p} \| \Pi_\tau \psi(k\tau) - \psi^k_\tau \|_{l_\tau^{q_1'}([0, (n-1)\tau]; L^{r_0})} \notag \\
			&\leq |I|^{1 - \frac{1}{q_1} - \frac{1}{q_0}} \| V_1 \|_{L^{p}} \| \Pi_\tau \psi(k\tau) - \psi^k_\tau \|_{l_\tau^{q_0}(I; L^{r_0})}. 
		\end{align}
		Similarly, we have for $Z_2$, 
		\begin{align}\label{eq:Z2_est}
			\| Z_2 \|_{l^q_\tau(I; L^r)} 
			&\lesssim \| \vphi_1(i \tau \Delta) \Pi_\tau (V_2 \Pi_\tau \psi(k\tau) - V_2 \psi^k_\tau) \|_{l_\tau^1([0, (n-1)\tau]; L^2)} \notag \\
			&\leq \| V_2 \|_{L^\infty} \| \Pi_\tau \psi(k\tau) - \psi^k_\tau \|_{l_\tau^1([0, (n-1)\tau]; L^2)} \notag \\
			&\leq |I| \| V_2 \|_{L^\infty} \| \Pi_\tau \psi(k\tau) - \psi^k_\tau \|_{l^\infty_\tau(I; L^2)}. 
		\end{align}
		The combination of \cref{eq:Z1_est,eq:Z2_est} completes the proof. 
	\end{proof}
	
	\subsection{Proof of \texorpdfstring{\cref{thm:main1,thm:main2}}{Theorems 1.1 and 1.2}}
	With all the estimates above, we are able to obtain the error estimate of the EWI \cref{eq:EWI_scheme}, and complete the proof of \cref{thm:main1,thm:main2}. 
	\begin{proof}[Proof of \cref{thm:main1}]
		Let $T_\ast<T$ be defined such that 
		\begin{equation}\label{eq:Tast}
			C_0\left( |T_\ast|^{1 - \frac{1}{q_1} - \frac{1}{q_0}} + |T_\ast| \right) \leq \frac{1}{4}, 
		\end{equation}
		where $C_0$ is the constant in \cref{prop:Z}. 
		
		We first present the estimate of $e^n$ for $0 \leq n \leq T_\ast/\tau$. The estimate for general $0 \leq n \leq T/\tau$ can be easily obtained by dividing $[0, T]$ into several short intervals with length less than $T_\ast$ and repeating the same process on each subinterval (see the proof of Theorems 2.1 and 2.2 in \cite{bao2025_singular} for more details).  
		
		From \cref{eq:error_eq}, using \cref{eq:E0_est,prop:E1_pg2,prop:E2,prop:Z}, for $(q, r) \in \{(q_0, r_0), (\infty, 2)\}$ and $n \tau \leq T_\ast$, we have, recalling \cref{eq:Tast}, 
		\begin{align}\label{eq:error_est}
			\| e_\tau \|_{l_\tau^q([0, n\tau]; L^r)} 
			&\leq \| \mathcal{E}_0 \|_{l_\tau^q([0, n\tau]; L^r)} + \| \mathcal{E}_1 \|_{l_\tau^q([0, n\tau]; L^r)} + \| \mathcal{E}_2 \|_{l_\tau^q([0, n\tau]; L^r)} + \| Z \|_{l_\tau^q([0, n\tau]; L^r)} \notag \\
			&\leq \widetilde{C}_0 \| e^0_\tau \|_{L^2} + C \tau + C_0((n\tau)^{1 - \frac{1}{q_1} - \frac{1}{q_0}} + n \tau) \| e_\tau \|_{X_\tau([0, n\tau])} \notag \\
			&\leq  \widetilde{C}_0 \| e^0_\tau \|_{L^2} + C \tau + \frac{1}{4} \| e_\tau \|_{X_\tau([0, n\tau])}. 
		\end{align} 
		Taking $(q, r) = (q_0, r_0) \text{ and } (\infty, 2)$ in \cref{eq:error_est}, and summing together, we have
		\begin{equation}\label{err1}
			\| e_\tau \|_{X_\tau([0, n\tau])} \leq 2\widetilde{C}_0\| e^0_\tau \|_{L^2} + C\tau + \frac{1}{2} \| e_\tau \|_{X_\tau([0, n\tau])}, 
		\end{equation}
		which implies
		\begin{equation}\label{err2}
			\| e_\tau \|_{X_\tau([0, n\tau])} \leq 4\widetilde{C}_0\| e^0_\tau \|_{L^2} + C \tau.
		\end{equation}
		Inserting \cref{err2} into \cref{eq:error_est} yields that
		\begin{equation}\label{err3}
			\| e_\tau \|_{l_\tau^\infty([0, n\tau]; L^2)} \leq 2\widetilde{C}_0 \| e^0_\tau \|_{L^2} + C \tau \lesssim \tau.  
		\end{equation}
		Then we have established the error estimate for $0 \leq n \leq T_\ast/\tau$. The proof can be completed with the discussion at the beginning of the proof. 
	\end{proof}
	
	The proof of \cref{thm:main2} follows the same lines as the proof of \cref{thm:main1} by replacing \cref{prop:E1_pg2} with \cref{prop:E1_pl2}, and thus is omitted for brevity. 
	
	\section{Extension to the nonlinear case}\label{sec:nonlinear}
	All the estimates in \cref{sec:3,sec:4} are presented for the Schr\"odinger equation with $\beta = 0$ in \cref{NLSE}. In this section, we explain how to extend the main results \cref{thm:main1,thm:main2} to the nonlinear case with $\beta \neq 0$. 
	
	\subsection{Extension to the NLSE \cref{NLSE}}
	For the NLSE \cref{NLSE} with $\beta \neq 0$, there are additional terms incorporating the nonlinearity in $\mathcal{E}_1$, $\mathcal{E}_2$, and ${Z}$ in the error equation \cref{eq:error_eq}, as well as in the estimate of $R$ \cref{eq:R_def} when substituting \cref{NLSE} in \cref{eq:R_decomp_1}. In fact, all those additional terms appearing in the local truncation errors, i.e. $\mathcal{E}_1$ and $\mathcal{E}_2$ can be handled in a similar way as for $V_2$ since we have $\| \psi \|_{L^\infty([0, T]; L^\infty)} < \infty$ for the exact solution $\psi$. Only the term appearing in $Z^n$ \cref{eq:Z_def} needs to be handled differently (but is still standard) since we need the control of the numerical solution to conclude the error estimate for the nonlinear equation. When $\beta \neq 0$, the additional term incorporating the nonlinearity in $Z_n$ takes the form
	\begin{equation}
		Z_{\rm nonl}^n := - i \tau \beta \sum_{k=0}^{n-1} \Stau((n-1)\tau - k \tau) \vphi_1(i \tau \Delta) \Pi_\tau (|\Pi_\tau \psi(t_k)|^2 \Pi_\tau\psi(t_k) - |\psi^k_\tau|^2\psi^k_\tau). 
	\end{equation}
	By H\"older's inequality, \cref{lem:dS2,lem:phi_1_Pi}, for any admissible pairs $(q, r)$ with $q \neq 2$,  
	\begin{align}\label{eq:stability_cubic}
		&\| Z_{\rm nonl} \|_{l^q_\tau([0, n\tau]; L^r)} \notag \\
		&\lesssim (\| \psi \|_{L^\infty([0, (n-1)\tau]; L^{2p})}^2 + \| \psi_\tau \|_{l^\infty_\tau([0, (n-1)\tau]; L^{2p})}^2) \| \Pi_\tau \psi - \psi_\tau \|_{l_\tau^{q_1'}([0, (n-1)\tau]; L^{r_0})} \notag \\
		&\lesssim |(n-1)\tau|^{1- \frac{1}{q_1} - \frac{1}{q_0}} (\| \psi \|_{L^\infty([0, (n-1)\tau]; L^{2p})}^2 + \| \psi_\tau \|_{l^\infty_\tau([0, (n-1)\tau]; L^{2p})}^2) \| \Pi_\tau \psi - \psi_\tau \|_{l_\tau^{q_0}([0, (n-1)\tau]; L^{r_0})}. 
	\end{align}
	Hence, to control the growth of the nonlinearity, we need to control the $L^{2p}$-norm of the numerical solution. Note that, when $\tau \leq \tau_0$ with $\tau_0 > 0 $ sufficiently small depending on the exact solution $\psi$, using Sobolev embedding and \cref{lem:Pi}, 
	\begin{align}\label{eq:induction}
		\| \psi^{n}_\tau \|_{L^{2p}} 
		&\leq \| \psi(t_n) \|_{L^{2p}} + \| \psi(t_n) - \Pi_\tau \psi(t_n) \|_{L^{2p}} + \| \Pi_\tau \psi(t_n) - \psi^n_\tau \|_{L^{2p}} \notag \\
		&\leq \| \psi(t_n) \|_{L^{2p}} + \| \, |\nabla|^{\frac{d}{2}(1 - \frac{1}{p})} (I - \Pi_\tau) \psi(t_n) \|_{L^2} + \| \, |\nabla|^{\frac{d}{2}(1 - \frac{1}{p})} ( \Pi_\tau \psi(t_n) - \psi^n_\tau) \|_{L^{2}} \notag \\
		&\leq \| \psi(t_n) \|_{L^{2p}} + C N_\tau^{\frac{d}{2}(1-\frac{1}{p})} N_\tau^{-(2-\alpha)} \| \psi(t_n) \|_{H^{2-\alpha}} + C N_\tau^{\frac{d}{2}(1-\frac{1}{p})} \| \Pi_\tau \psi(t_n) - \psi^n_\tau \|_{L^{2}} \notag \\
		&\leq \| \psi(t_n) \|_{L^{2p}} + \frac{1}{2} + C \tau^{-\frac{d}{4}(1 - \frac{1}{p})} \tau^{\gamma(d, p)} \leq  \| \psi \|_{L^\infty([0, T]; L^{2p})} + 1, 
	\end{align}
	where $\gamma(d, p)$ is the obtained convergence order in \cref{thm:main1,thm:main2}, and we use that 
	\begin{equation*}
		\frac{d}{2}(1-\frac{1}{p}) \leq \frac{d}{2} < 2 - \alpha, \quad \gamma(d, p) > \frac{d}{4}(1 - \frac{1}{p}), \quad p \geq 1, \ p > \frac{d}{2}, \quad 1 \leq d \leq 3. 
	\end{equation*}
	This shows that at each time step, the $L^{2p}$-norm of the numerical solution can be controlled by the exact solution and the obtained error bound with inverse estimates. Then the error estimate can be completed using the standard induction arguments as in \cite{bao2025_singular}. 
	
	\subsection{Extension to other nonlinearity}
	We can also consider the nonlinear Schr\"odinger equation with a general power-type nonlinearity given by
	\begin{equation}\label{NLSE-general}
		i \partial_t \psi(\vx, t) = -\Delta \psi(\vx, t) + V(\vx) \psi(\vx, t) + \beta |\psi(\vx, t)|^{2\sigma}\psi(\vx, t), \quad \vx \in \R^d, \quad \sigma \in \R^+. 
	\end{equation}
	As is discussed in Section 1.2 of \cite{wu2025_singularhd} and Section 1.2 of \cite{wu2025_singular1d}, the NLSE \cref{NLSE-general} with the same singular potential $V$ satisfying \cref{eq:V_assumption} is still well-posed in $H^{2-\alpha}(\R^d)$ if $\sigma > (1-\alpha)/2$. Heuristically, the lower bound of $\alpha$ is required to ensure the nonlinear term $|\psi|^{2\sigma}\psi$ satisfies the basic regularity requirement: $|\psi|^{2\sigma}\psi \in H^{2-\alpha}$ if $\psi \in H^{2-\alpha}$. Then the same space-time estimates as in \cref{eq:regularity} hold for \cref{NLSE-general}, and we can extend \cref{thm:main1,thm:main2} to \cref{NLSE-general}, under additional assumtions on the nonlinearity $\sigma$ as shown below. 
	
	For the error estimate of the EWI applied to \cref{NLSE-general}, we have the following result: \cref{thm:main1,thm:main2} hold for the NLSE \cref{NLSE-general} if
	\begin{equation}\label{eq:sigma}
		\frac{1-\alpha}{2} < \sigma \left\{
		\begin{aligned}
			&<\infty, && d=1, \ p \geq 1 \text{ or } d=2, \  p \geq \frac{4}{3} \text{ or } d= 3, \  p \geq \frac{9}{5}, \\
			&< \frac{3}{4 - 3p}, && d=2, \  1 < p< \frac{4}{3}, \\
			&< \frac{3}{18 - 10p}, && d= 3, \   \frac{3}{2} < p < \frac{9}{5}.  
		\end{aligned}
		\right.
	\end{equation}
	
	In the following, we briefly explain how to obtain this result. Similar to the cubic nonlinearity case, the estimates related to the local truncation errors, i.e., $\mathcal{E}_1$ and $\mathcal{E}_2$ in \cref{eq:error_eq} remain similar when $\sigma > (1-\alpha)/2$, and only the stability estimate for $Z^n$ \cref{eq:Z_def} need to be modified. We consider two cases: 
	\begin{enumerate}[(i)]
		\item  When $d=1$, or $d=2$, $p \geq 4/3$, or $d=3$, $p \geq 9/5$, the convergence order in $L^2$-norm is of order $O(\tau^{\gamma(d, p)})$ where $\gamma(d, p)$ is the corresponding order obtained in \cref{thm:main1,thm:main2} and we verify that $\gamma(d, p)>d/4$. This thus allows the use of inverse estimates to control the $L^\infty$-norm of the numerical solution. In this case, \cref{thm:main1,thm:main2} hold for any $\sigma > (1-\alpha)/2$. 
		
		\item When $d=2$, $1 < p < 4/3$ or $d=3$, $3/2<p<9/5$, we further consider three regimes: 
		\begin{enumerate}[--]
			\item If $\sigma > \frac{1}{p}$, we can have an analog of \cref{eq:stability_cubic,eq:induction} by replacing $2p$ with $2 \sigma p$. Then we need in addition
			\begin{equation*}
				\gamma(d, p) > \frac{d}{4} ( 1 - \frac{1}{\sigma p})  \impliedby \left\{
				\begin{aligned}
					&\sigma < \frac{1}{4-3p}, &&d=2, \\
					&\sigma < \frac{3}{18 - 10p}, && d = 3.
				\end{aligned}
				\right. 
			\end{equation*}
			
			\item If $\frac{1}{d} \leq \sigma \leq \frac{1}{p}$, \cref{eq:stability_cubic} should be adapted by
			\begin{align*}\label{eq:stability_power}
				&\| Z_{\rm nonl} \|_{l^q_\tau([0, n\tau]; L^r)} \\
				&\lesssim (\| \psi \|_{L^\infty([0, (n-1)\tau]; L^{2})}^{2\sigma} + \| \psi_\tau \|_{l^\infty_\tau([0, (n-1)\tau]; L^{2})}^{2\sigma}) \| \Pi_\tau \psi - \psi_\tau \|_{l_\tau^{q_1'}([0, (n-1)\tau]; L^{\tilde r_0})} \\
				&\lesssim |(n-1)\tau|^{1- \frac{1}{q_1} - \frac{1}{\tilde q_0}} (\| \psi \|_{L^\infty([0, (n-1)\tau]; L^{2})}^{2\sigma} + \| \psi_\tau \|_{l^\infty_\tau([0, (n-1)\tau]; L^{2})}^{2\sigma}) \| \Pi_\tau \psi - \psi_\tau \|_{l_\tau^{\tilde q_0}([0, (n-1)\tau]; L^{\tilde r_0})},  
			\end{align*}
			where $(\tilde q_0, \tilde r_0)$ is an admissible pair with $\tilde r_0$ determined by
			\begin{equation*}
				\frac{1}{\tilde r_0} + \frac{1}{r_1} + \sigma = 1. 
			\end{equation*}
			Since $\frac{1}{d} \leq \sigma \leq \frac{1}{p}$, we have $2 < \tilde r_0 \leq r_0$, and, for some $0 < \beta \leq 1$, 
			\begin{equation*}
				\frac{1}{\tilde r_0} = \frac{1 - \beta}{2} + \frac{\beta}{r_0}, \quad \frac{1}{\tilde q_0} = \frac{1-\beta}{\infty} + \frac{\beta}{q_0},  
			\end{equation*}
			which implies, by interpolation, 
			\begin{align}\label{eq:interpolation}
				\| \Pi_\tau \psi - \psi_\tau \|_{l_\tau^{\tilde q_0}([0, (n-1)\tau]; L^{\tilde r_0})} 
				&\leq \| \Pi_\tau \psi - \psi_\tau \|_{l_\tau^{\infty}([0, (n-1)\tau]; L^{2})}^{1-\beta} \| \Pi_\tau \psi - \psi_\tau \|_{l_\tau^{q_0}([0, (n-1)\tau]; L^{r_0})}^\beta \notag \\
				&\leq \| \Pi_\tau \psi - \psi_\tau \|_{X_\tau([0, (n-1)\tau])}.  
			\end{align}
			Hence, it suffices to establish a uniform $L^2$-norm bound of the numerical solution, which is a direct result of the $L^2$-norm error bound using an induction argument similar to \cref{eq:induction}. 
			
			\item If $\frac{1-\alpha}{2} < \sigma < \frac{1}{d}$, we modify \cref{eq:stability_cubic} as follows: choosing $(\tilde q, \tilde r) = (\frac{1}{\sigma}, \frac{2}{1-2\sigma})$ when $d=2$, and $(\tilde q, \tilde r) = (4, 3)$ when $d=3$ in \cref{lem:dS2}, we have , for any $(q, r)$ admissible, 
			\begin{align*}\label{eq:stability_power_2}
				&\| Z_{\rm nonl} \|_{l^q_\tau([0, n\tau]; L^r)} \\
				&\lesssim (\| \psi \|_{L^\infty([0, (n-1)\tau]; L^{2})}^{2\sigma} + \| \psi_\tau \|_{l^\infty_\tau([0, (n-1)\tau]; L^{2})}^{2\sigma}) \| \Pi_\tau \psi - \psi_\tau \|_{l_\tau^{\tilde q'}([0, (n-1)\tau]; L^{\tilde r_0})} \\
				&\lesssim |(n-1)\tau|^{1- \frac{1}{\tilde q} - \frac{1}{\tilde q_0}} (\| \psi \|_{L^\infty([0, (n-1)\tau]; L^{2})}^{2\sigma} + \| \psi_\tau \|_{l^\infty_\tau([0, (n-1)\tau]; L^{2})}^{2\sigma}) \| \Pi_\tau \psi - \psi_\tau \|_{l_\tau^{\tilde q_0}([0, (n-1)\tau]; L^{\tilde r_0})},  
			\end{align*}
			where $(\tilde q_0, \tilde r_0)$ is an admissible pair with $\tilde r_0$ satisfying
			\begin{equation*}
				\frac{1}{\tilde r_0} + \frac{1}{\tilde r} + \sigma = 1. 
			\end{equation*}
			It can be checked that $2 \leq \tilde r_0 \leq r_0$. Using \cref{eq:interpolation}, again, we only need to establish a uniform $L^2$-norm bound of the numerical solution, which can be obtained by the $L^2$-norm error bound and an induction argument as before. 
		\end{enumerate}
	\end{enumerate}
	Then the extension of \cref{thm:main1,thm:main2} to \cref{NLSE-general} under \cref{eq:sigma} is completed. 
	
	\section{Numerical results}\label{sec:numer}
	In this section, we present numerical results to verify the error estimates in \cref{thm:main1,thm:main2} for the EWI \cref{eq:EWI_scheme} applied to NLSE \cref{NLSE} with singular potentials in 1D, 2D, and 3D. 
	
	
	To implement the EWI, we choose a computation domain $\Omega = (-L, L)^d$ with $L$ large enough such that the truncation error is negligible. On the bounded domain $\Omega$, we impose the periodic boundary conditions and use the standard Fourier spectral method for the spatial discretization. We define the following error functions:
	\begin{equation*}
		e_{L^2}(t_n) = \| \psi(t_n) - \psi^n \|_{L^2}, \qquad e_{H^1}(t_n) = \| \psi(t_n) - \psi^n \|_{H^1}, \quad 0 \leq n \leq T/\tau. 
	\end{equation*}
	
	We first consider the 1D case and choose $\Omega = (-16, 16)$ and $\beta = -1$. As the optimal first-order convergence in the $L^2$-norm has already been verified for $L^2$-potentials in \cite{bao2025_singular}, we focus on more singular potentials in $L^p$ with $1 \leq p < 2$. As a typical example, we consider the (periodic) Dirac delta potential in 1D:
	\begin{equation}\label{eq:delta}
		\delta(x) = \frac{1}{2L} \sum_{l \in \mathbb{Z}} e^{i \mu_l x}, \quad \mu_l = \frac{\pi l}{L}, \quad x \in \Omega. 
	\end{equation}
	Although \cref{eq:delta} is not an $L^1$ function, it is the limit of $L^1$-functions, and we (roughly) have $\alpha = \frac{1}{2}$. In the computation, we approximate $-\delta(x)$ by a truncated Fourier series
	\begin{equation}\label{eq:V_delta}
		V(x) = -\frac{1}{2L} \sum_{|l| \leq N_\text{ref}} e^{i \mu_l x}, \quad x \in \Omega,  
	\end{equation}
	where $N_\text{ref} = 2^{16}$ is chosen much larger than the spatial degree of freedom. We use the following stationary solution of the NLSE with Dirac delta potential in $\R$ as the initial data:
	\begin{equation*}
		\psi_0(x) = \sqrt{2}
		\operatorname{sech}\left(|x|+\operatorname{arctanh}\left(\frac{1}{2}\right)\right), \quad x \in \Omega.
	\end{equation*} 
	The exact solution is given by $\psi(x, t) = e^{i t} \psi_0(x)$. Note that $\Omega$ is chosen large enough such that the domain truncation error can be ignored. 
	
	In the numerical simulation, we fix the mesh size $h = 2^{-9}$ and compute up to $T=1$. The errors in $L^2$- and $H^1$-norms are shown in \cref{fig:conv_dt_1D2D} (a). We observe that the $L^2$-norm convergence is of half order, which agrees with the error estimate in \cref{thm:main2} very well, indicating that the error bound is sharp even for highly singular potentials in 1D. This observation also confirms that the dominant error in 1D comes from the spatial discretization (filteration) rather than the temporal discretization. 
	
	\begin{figure}[htbp]
		\centering
		{\includegraphics[width=0.475\textwidth]{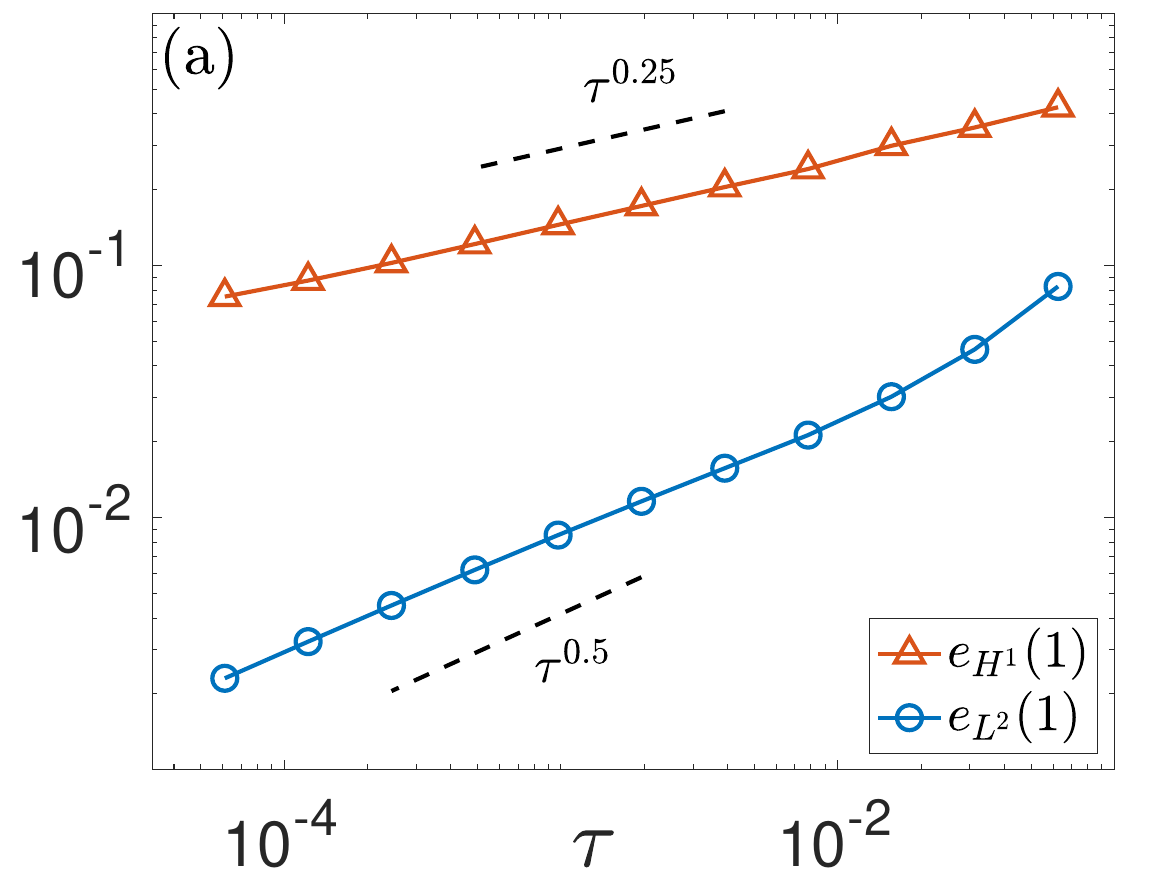}}\hspace{1em}
		{\includegraphics[width=0.475\textwidth]{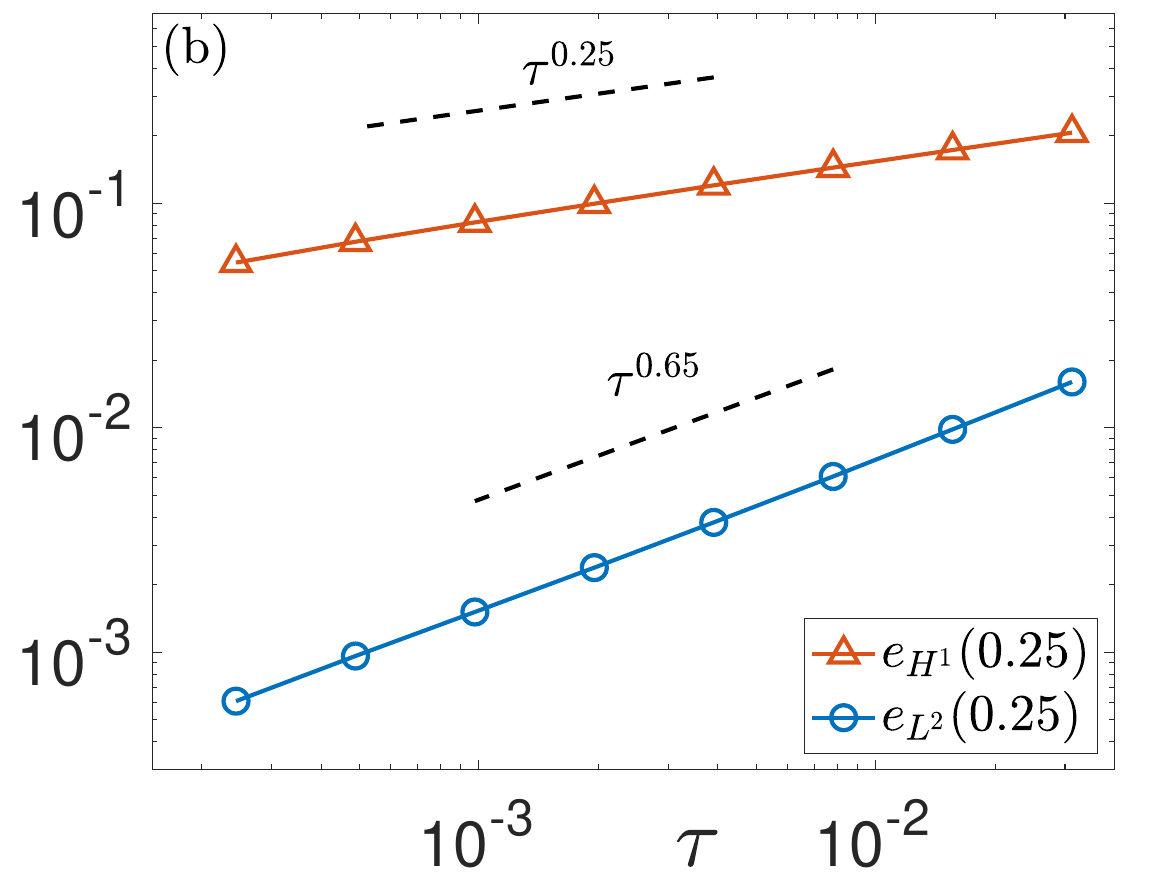}}
		\caption{Errors in $L^2$- and $H^1$-norms of the EWI for the NLSE \cref{NLSE}: (a) 1D case, and (b) 2D case}
		\label{fig:conv_dt_1D2D}
	\end{figure}
	
	Next we consider the 2D case with $\Omega = (-8, 8)^2$ and $\beta = 0$. Again, we focus on highly singular potentials $V \in L^p$ with $1 < p <2$ as the $L^2$ case has been treated in \cite{bao2025_singular}. We define the singular potential in the Fourier space by (with $L=8$)
	\begin{equation}\label{eq:V_2D}
		V(\vx) = -\frac{1}{(2L)^2} \sum_{\mathbf 0 \neq \mathbf l \in \mathbb{Z}^2} |\mu_\textbf{l}|^{-\frac{1}{2}} e^{i \mu_\textbf{l}\cdot \vx}, \quad \mu_{\mathbf l} = \frac{\pi \mathbf l}{L}, \quad \vx \in \Omega.
	\end{equation}
	Recall that the Fourier transform of $|\vx|^{-\frac{3}{2}}$ is given by $C|\vxi|^{-\frac{1}{2}}$. Thus, near the origin, $V(\vx) \sim |\vx|^{-\frac{3}{2}}$, and we have $V \in L^{\frac{4}{3}^-}$ and  $\alpha = \frac{1}{2}^+$. 
	The initial data is chosen as a standard Gaussian with $\psi_0(\vx) = e^{- \frac{|\vx|^2}{2}}$. In the numerical simulation, we choose the final time $T = 0.25$ and the mesh size $h = 2^{-8}$ in each direction. The reference solution is obtained using a time step size $\tau = 10^{-5}$. The numerical result is demonstrated in \cref{fig:conv_dt_1D2D} (b). The observed convergence rate in the $L^2$-norm is about $0.65$ order, which is slightly higher than the half-order convergence proved in \cref{thm:main2}. It remains unclear whether the error bound in \cref{thm:main2} is optimal for highly singular potentials in 2D. On the other hand, the observed $0.65$-order rate is still slower than the temporal convergence rate proved in \cref{prop:E2}, which is of $0.75$ order. This indicates again that the overall error is still dominated by the spatial filteration error (i.e. \cref{eq:E1}). 
	
	Finally, we consider the 3D case with $\Omega = (-8, 8)^3$ and $\beta = 1$. The singular potential is chosen similarly to \cref{eq:V_2D} by (with $L=8$)
	\begin{equation}\label{eq:V_3D}
		V(\vx) = -\frac{1}{(2L)^3} \sum_{\mathbf 0 \neq \mathbf l \in \mathbb{Z}^3} |\mu_\textbf{l}|^{-\gamma} e^{i \mu_\textbf{l}\cdot \vx}, \quad \mu_{\mathbf l} = \frac{\pi \mathbf l}{L}, \quad \vx \in \Omega, \quad \gamma > 0. 
	\end{equation}
	We consider two cases: (i) an $L^{2^-}$-potential with $\gamma = 3/2$ in \cref{eq:V_3D} and $\alpha = 0^+$, and (ii) an $L^{\frac{12}{7}^-}$-potential with $\gamma = \frac{5}{4}$ in \cref{eq:V_3D} and $\alpha = {\frac{1}{4}}^+$. In both cases, we take the standard Gaussian initial datum $\psi_0(\vx) = e^{- \frac{|\vx|^2}{2}}$ and choose the final time $T = 0.125$. The mesh size is fixed as $h = 2^{-5}$ in each direction, and the reference solution is obtained with a time step size $\tau = 10^{-4}$. The numerical results are presented in \cref{fig:conv_dt_3D} (a) and (b) for the two cases, respectively. From \cref{fig:conv_dt_3D} (a), we see that under $L^2$-potential, the $L^2$-norm convergence is first-order, which verifies our error estimates in \cref{thm:main1,thm:main2}. For more singular $L^\frac{12}{7}$-potential, the observed convergence order is (higher than) 0.65 order, which is also higher than the predicted $\frac{5}{8}$ order in \cref{thm:main2}. We remark that the convergence order for highly singular $L^\frac{12}{7}$-potential shown in \cref{fig:conv_dt_3D} (b) does not yet appear fully stable. However, the simulation for smaller time steps also requires refining the spatial mesh (due to the filter $\Pi_\tau$), which thus requires substantial memory cost for this 3D test problem. Therefore, the sharpness of the error bound in \cref{thm:main2} for 3D highly singular potential remains unclear and hard to verify numerically. 
	\begin{figure}[htbp]
		\centering
		{\includegraphics[width=0.475\textwidth]{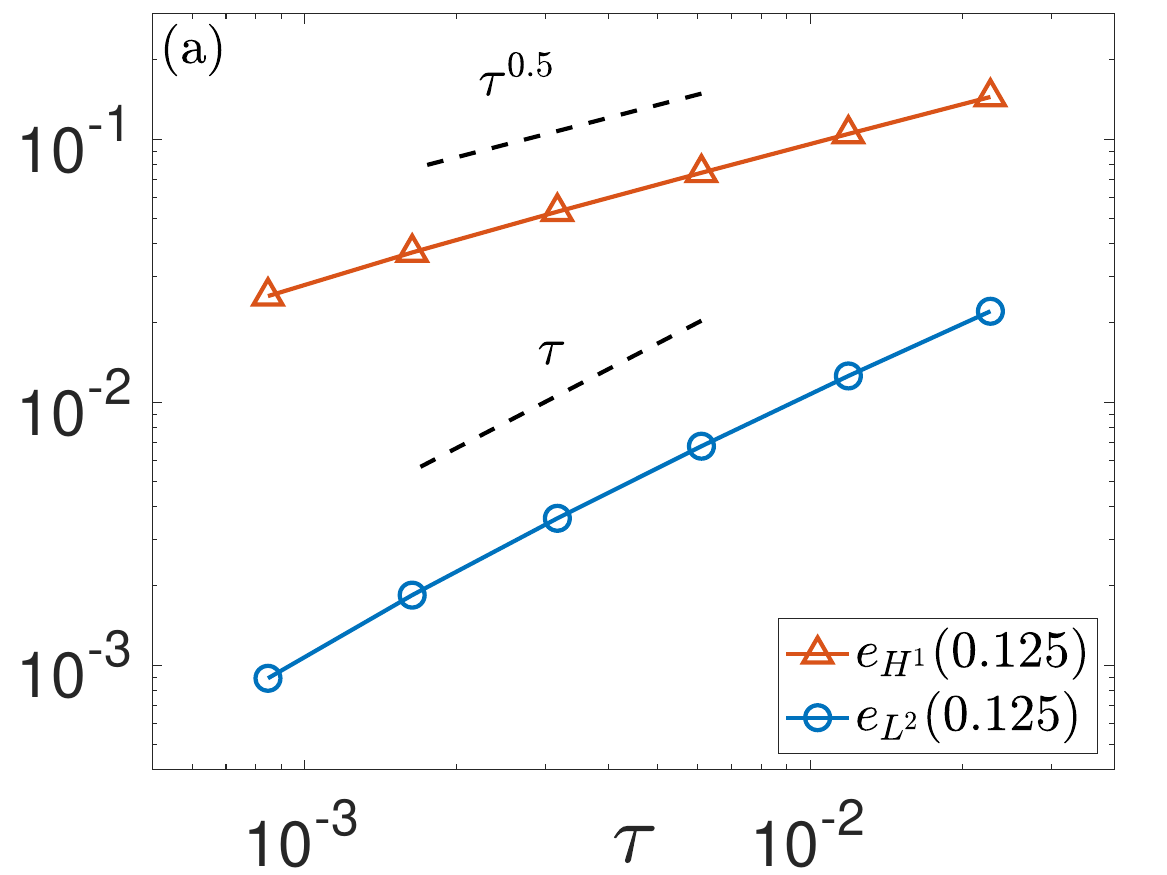}}\hspace{1em}
		{\includegraphics[width=0.475\textwidth]{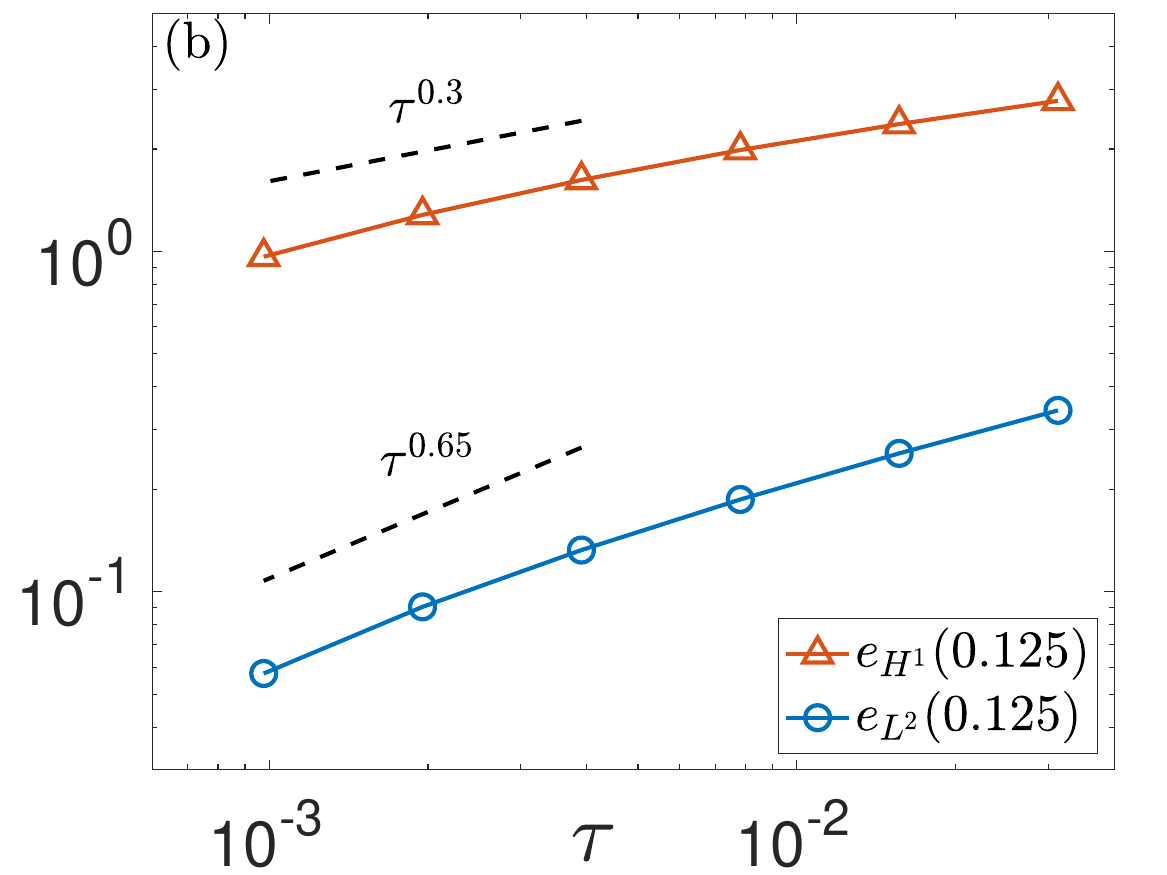}}
		\caption{Errors in $L^2$- and $H^1$-norms of the EWI for the NLSE \cref{NLSE}: (a) $L^{2^-}$-potential, and (b) $L^{\frac{12}{7}^-}$-potential}
		\label{fig:conv_dt_3D}
	\end{figure}
	
	\section{Conclusion}\label{sec:conclusion}
	We established error estimates for a filtered exponential wave integrator (EWI) for the nonlinear Schr\"odinger equation with highly singular potentials in $L^p(\R^d) + L^\infty(R^d)$ with $p > \frac{d}{2}$, $p \geq 1$, with the spatial dimension $d=1, 2, 3$. Our error estimate has pushed the regularity requirement on the potential to the threshold regularity such that the NLSE remains well-posed. For $L^{2^+}_\text{loc}$-potentials, we obtained an optimal first-order $L^2$-norm error bound in all three dimensions $d=1,2,3$. The optimality is two-fold: (i) in the presence of an $L^2$-potential, the first-order $L^2$-norm convergence of the EWI is optimal, and (ii) to achieve first-order $L^2$-norm convergence of the EWI, $L^2$-regularity of the potential is the minimal regularity required. Furthermore, we extended the analysis to more singular potentials with $\frac{d}{2}<p\leq2$ (and $p \geq 1$), and obtain (roughly) $(1 - \alpha)$-order convergence in one and two dimensions, and $(1 - \frac{3}{2}\alpha)$-order convergence in three dimensions, where $\alpha := d(\frac{1}{p} - \frac{1}{2})$ if $p >1$ and $\alpha := \frac{1}{2}^+$ if $p=1, d=1$. Extensive numerical results are reported to verify the error estimates and show the optimality of the error bound.  
	
	\section*{Acknowledgment}
	This work was partially supported by the Ministry of Education of Singapore under its AcRF Tier 1 funding grant A-8003584-00-00 (W. Bao). This work was supported in part by the National Natural Science Foundation of China Grant No. 12494544 (Y. Wu).
	

\end{document}